\documentclass[12pt]{amsart}
\usepackage{amsmath,amssymb,amsbsy,amsfonts,amsthm,latexsym,
                        amsopn,amstext,amsxtra,euscript,amscd,mathrsfs,color,bm}
                    
\usepackage{float} 
\usepackage[english]{babel}
\usepackage{mathtools}
\usepackage[textsize=tiny]{todonotes}
\usepackage{url}
\usepackage[colorlinks,linkcolor=blue,anchorcolor=blue,citecolor=blue]{hyperref}
\usepackage{nicefrac}
\usepackage{enumitem}

\usepackage[a4paper,  left=2cm, right=2cm, top=3.5 cm, bottom=3.5cm]{geometry}

\usepackage[norefs,nocites]{refcheck}

\usepackage[norefs,nocites]{refcheck}
\usepackage[norefs,nocites]{refcheck}
     
  \usepackage{soul}
\begin{document}
%
%\renewcommand*{\backref}[1]{}
%\renewcommand*{\backrefalt}[4]{%
%    \ifcase #1 (Not cited.)%
%    \or        (p.\,#2)%
%    \else      (pp.\,#2)%
%    \fi}     
   % \linenumbers
%%%%%%%%%%%%%%%%%%%%%%%%%%%%%%%%%%%%%%%%%%%%%%%%%%%%%%%%%%%%%%%%%%%%%%
\newtheorem{theorem}{Theorem}
\newtheorem{lemma}[theorem]{Lemma}
\newtheorem{example}[theorem]{Example}
\newtheorem{algol}{Algorithm}
\newtheorem{corollary}[theorem]{Corollary}
\newtheorem{prop}[theorem]{Proposition}
\newtheorem{proposition}[theorem]{Proposition}
\newtheorem{problem}[theorem]{Problem}
\newtheorem{conj}[theorem]{Conjecture}

\theoremstyle{remark}
\newtheorem{definition}[theorem]{Definition}
\newtheorem{question}[theorem]{Question}
\newtheorem{remark}[theorem]{Remark}
\newtheorem*{acknowledgement}{Acknowledgements}

\newtheorem*{Thm*}{Theorem}
\newtheorem{Thm}{Theorem}[section]
\renewcommand*{\theThm}{\Alph{Thm}}

\numberwithin{equation}{section}
\numberwithin{theorem}{section}
\numberwithin{table}{section}
\numberwithin{figure}{section}

\allowdisplaybreaks

%%%%%%%%%%%%%%%%%%%%%%%%%%%%%%%%%%%%%%%%%%%%%%%%%%%%%%%%%%%%%%%%%%%%%%
\definecolor{olive}{rgb}{0.3, 0.4, .1}
\definecolor{dgreen}{rgb}{0.,0.5,0.}

\def\cc#1{\textcolor{red}{#1}} 

\definecolor{dgreen}{rgb}{0.,0.6,0.}
\def\tgreen#1{\begin{color}{dgreen}{\it{#1}}\end{color}}
\def\tblue#1{\begin{color}{blue}{\it{#1}}\end{color}}
\def\tred#1{\begin{color}{red}#1\end{color}}
\def\tmagenta#1{\begin{color}{magenta}{\it{#1}}\end{color}}
\def\tNavyBlue#1{\begin{color}{NavyBlue}{\it{#1}}\end{color}}
\def\tMaroon#1{\begin{color}{Maroon}{\it{#1}}\end{color}}

\def\ccr#1{\textcolor{red}{#1}}
\def\ccm#1{\textcolor{magenta}{#1}}
\def\cco#1{\textcolor{orange}{#1}}

%\newcommand{\commA}[1]{\marginpar{%
%\begin{color}{blue}
%\vskip-\baselineskip %raise the marginpar a bit
%\raggedright\footnotesize
%\itshape\hrule \smallskip AO: #1\par\smallskip\hrule\end{color}}\ignorespaces}
%
%\newcommand{\commI}[1]{\marginpar{%
%\begin{color}{magenta}
%\vskip-\baselineskip %raise the marginpar a bit
%\raggedright\footnotesize
%\itshape\hrule \smallskip IS: #1\par\smallskip\hrule\end{color}}\ignorespaces}

%%%%%%%%%%%%%%%%%%%%%%%%%%%%%%%%%%%%%%%%%%%%%%%%%%%%%%%%%%%%%%%%%%%%%%

%%%%%%%%%%%%%%%%%%%%%%%%%%%%%%%%%%%%%%%%%%%%%%%%%%%%%%%%%%%%%%%%%%%%%%
%% %% Macros that are not used or that were defined twice
%% \def\xxx{\vskip5pt\hrule\vskip5pt}
%% \def\Cmt#1{\underline{{\sl Comments:}} {\it{#1}}}
%% \newcommand{\Modp}[1]{
%% \begin{color}{blue}
%%  #1\end{color}}
%% \def\bl#1{\begin{color}{blue}#1\end{color}} 
%% \def\red#1{\begin{color}{red}#1\end{color}} 
%\newcommand{\eqname}[1]{\tag{#1}}% Tag equation with name
%% JS: It is an abomination to redefine \(,\),\[,\]. They are STANDARD TeX macros
%% for delineating equations!!!
%% \def\({\left(}
%% \def\){\right)}
%% \def\[{\left[}
%% \def\]{\right]}
%% \def\gen#1{{\left\langle#1\right\rangle}}
%% \def\genp#1{{\left\langle#1\right\rangle}_p}
%% \def\genPs{{\left\langle P_1, \ldots, P_s\right\rangle}}
%% \def\genPsp{{\left\langle P_1, \ldots, P_s\right\rangle}_p}
%% \def\eq{\e_q}
%% \def\fh{{\mathfrak h}}
%% \def\fl#1{\left\lfloor#1\right\rfloor}
%% \def\rf#1{\left\lceil#1\right\rceil}
 \def\mand{\qquad\mbox{and}\qquad}
%% \def\jt{\tilde\jmath}
%% \def\ellmax{\ell_{\rm max}}
%% \def\rank#1{\mathrm{rank}#1} 
%% \def\m{{\rm m}}
%% \def\ch{\hat{h}}
%% \def\GL{{\rm GL}}
%% \def\Orb{\operatorname{Orb}}
%% \def\Per{\operatorname{Per}}
%% \def\Preper{\operatorname{Preper}}
%% \def\PGL{\operatorname{PGL}}
%% \def\tors{\operatorname{tors}}
%% \def\Gal{{\operatorname{Gal}}}
%% \newcommand{\Ch}{{\operatorname{Ch}}}
%% \newcommand{\Elim}{{\operatorname{Elim}}}
%% \newcommand{\proj}{{\operatorname{proj}}}
%% \newcommand{\h}{{\operatorname{\mathrm{h}}}}
%% \newcommand{\hh}{\mathrm{h}}
%% \newcommand{\bfalpha}{{\boldsymbol{\alpha}}}
%% \newcommand{\bfomega}{{\boldsymbol{\omega}}}
%%%%%%%%%%%%%%%%%%%%%%%%%%%%%%%%%%%%%%%%%%%%%%%%%%%%%%%%%%%%%%%%%%%%%%

%%%%%%%%%%%%%%%%%%%%%%%%%
% Alphabet calligraphic %
%%%%%%%%%%%%%%%%%%%%%%%%%
\def\cA{{\mathcal A}}
\def\cB{{\mathcal B}}
\def\cC{{\mathcal C}}
\def\cS{{\mathcal D}}
\def\cH{{\mathcal H}}
\def\cI{{\mathcal I}}
\def\cJ{{\mathcal J}}
\def\cK{{\mathcal K}}
\def\cL{{\mathcal L}}
\def\cM{{\mathcal M}}
\def\cN{{\mathcal N}}
\def\cO{{\mathcal O}}
\def\cP{{\mathcal P}}
\def\cQ{{\mathcal Q}}
\def\cR{{\mathcal R}}
\def\cS{{\mathcal S}}
\def\cT{{\mathcal T}}
\def\cU{{\mathcal U}}
\def\cV{{\mathcal V}}
\def\cW{{\mathcal W}}
\def\cX{{\mathcal X}}
\def\cY{{\mathcal Y}}
\def\cZ{{\mathcal Z}}

\def\C{\mathbb{C}}
\def\F{\mathbb{F}}
\def\K{\mathbb{K}}
\def\Z{\mathbb{Z}}
\def\R{\mathbb{R}}
\def\Q{\mathbb{Q}}
\def\N{\mathbb{N}}
\def\M{\mathrm{M}}
\def\L{\mathbb{L}}
\def\M{{\normalfont\textsf{M}}} %%% MODIFIED BY DRP
\def\U{\mathbb{U}}
\def\P{\mathbb{P}}
\def\A{\mathbb{A}}
\def\fp{\mathfrak{p}}
\def\fq{\mathfrak{q}}
\def\n{\mathfrak{n}}
\def\X{\mathcal{X}}
\def\x{\textrm{\bf x}}
\def\w{\textrm{\bf w}}
\def\ovQ{\overline{\Q}}
\def \Kab{\K^{\mathrm{ab}}}
\def \Qab{\Q^{\mathrm{ab}}}
\def \Qtr{\Q^{\mathrm{tr}}}
\def \Kc{\K^{\mathrm{c}}}
\def \Qc{\Q^{\mathrm{c}}}
\def\ZK{\Z_\K}
\def\ZKS{\Z_{\K,\cS}}
\def\ZKSf{\Z_{\K,\cS_{f}}}
\def\RSf{R_{\cS_{f}}}
\def\RTf{R_{\cT_{f}}}
\def \wH{{\mathrm H}}

\def\S{\mathcal{S}}
\def\vec#1{\mathbf{#1}}
\def\ov#1{{\overline{#1}}}
\def\sign{{\operatorname{sign}}}
\def\Gm{\G_{\textup{m}}}
\def\fA{{\mathfrak A}}
\def\fB{{\mathfrak B}}

\def \GL{\mathrm{GL}}
\def \Mat{\mathrm{Mat}}

\def\house#1{{%
    \setbox0=\hbox{$#1$}
    \vrule height \dimexpr\ht0+1.4pt width .5pt depth \dp0\relax
    \vrule height \dimexpr\ht0+1.4pt width \dimexpr\wd0+2pt depth \dimexpr-\ht0-1pt\relax
    \llap{$#1$\kern1pt}
    \vrule height \dimexpr\ht0+1.4pt width .5pt depth \dp0\relax}}

%%%%%%%%%%%%%%%%%%%%%%%%%%%%%%%%%%%%%%%%%%%%%%%%%%%%%%%%%%%%%%%%%%%%%%

%%%%%%%% Set Up Environment for Notation %%%%%%%%%%%%%%
% This is currently set to allow quite wide items to be defined
\newenvironment{notation}[0]{%
  \begin{list}%
    {}%
    {\setlength{\itemindent}{0pt}
     \setlength{\labelwidth}{1\parindent}
     \setlength{\labelsep}{\parindent}
     \setlength{\leftmargin}{2\parindent}
     \setlength{\itemsep}{0pt}
     }%
   }%
  {\end{list}}

%%%%%%%% Set Up Environment for Parts in Theorems %%%%%%%%%%%%%%
\newenvironment{parts}[0]{%
  \begin{list}{}%
    {\setlength{\itemindent}{0pt}
     \setlength{\labelwidth}{1.5\parindent}
     \setlength{\labelsep}{.5\parindent}
     \setlength{\leftmargin}{2\parindent}
     \setlength{\itemsep}{0pt}
     }%
   }%
  {\end{list}}
% Use \Part{(a)}, instead of \item[(a)], to ensure upright font
\newcommand{\Part}[1]{\item[\upshape#1]}

%%%%%%%% Set Up Macro for Cases %%%%%%%%%%%%%%
\def\Case#1#2{%
\smallskip\paragraph{\textbf{\boldmath Case #1: #2.}}\hfil\break\ignorespaces}

\def\Subcase#1#2{%
\smallskip\paragraph{\textit{\boldmath Subcase #1: #2.}}\hfil\break\ignorespaces}

%%%%%%%%%%%%%%%%%%
% Greek Alphabet %
%%%%%%%%%%%%%%%%%%
\renewcommand{\a}{\alpha}
\renewcommand{\b}{\beta}
\newcommand{\g}{\gamma}
\newcommand{\bnu}{\bm{\nu}}
\newcommand{\bk}{\bm{k}}
\renewcommand{\d}{\Delta}
\newcommand{\e}{\epsilon}
\newcommand{\f}{\varphi}
\newcommand{\fhat}{\hat\varphi}
\newcommand{\bfphi}{{\boldsymbol{\f}}}
\renewcommand{\l}{\lambda}
\renewcommand{\k}{\kappa}
\newcommand{\lhat}{\hat\lambda}
\newcommand{\bfmu}{{\boldsymbol{\mu}}}
\renewcommand{\o}{\omega}
\renewcommand{\r}{\rho}
\newcommand{\rbar}{{\bar\rho}}
\newcommand{\s}{\sigma}
\newcommand{\sbar}{{\bar\sigma}}
\renewcommand{\t}{\tau}
\newcommand{\z}{\zeta}

%\newcommand{\D}{\Delta}
%\newcommand{\G}{\Gamma}
%\newcommand{\F}{\Phi}
%\renewcommand{\L}{\Lambda}

%%%%%%%%%%%%%%%%%%%%
% Fraktur Alphabet %
%%%%%%%%%%%%%%%%%%%%
\newcommand{\ga}{{\mathfrak{a}}}
\newcommand{\gb}{{\mathfrak{b}}}
\newcommand{\gn}{{\mathfrak{n}}}
\newcommand{\gp}{{\mathfrak{p}}}
\newcommand{\gP}{{\mathfrak{P}}}
\newcommand{\gq}{{\mathfrak{q}}}
\newcommand{\h}{{\mathfrak{h}}}
%%%%%%%%%%%%%%%%%%%
% Barred Alphabet %
%%%%%%%%%%%%%%%%%%%
\newcommand{\Abar}{{\bar A}}
\newcommand{\Ebar}{{\bar E}}
\newcommand{\kbar}{{\bar k}}
\newcommand{\Kbar}{{\bar K}}
\newcommand{\Pbar}{{\bar P}}
\newcommand{\Sbar}{{\bar S}}
\newcommand{\Tbar}{{\bar T}}
\newcommand{\gbar}{{\bar\gamma}}
\newcommand{\lbar}{{\bar\lambda}}
\newcommand{\ybar}{{\bar y}}
\newcommand{\phibar}{{\bar\f}}

%%%%%%%%%%%%%%%%%%%%%%%%%
% Calligraphic Alphabet %
%%%%%%%%%%%%%%%%%%%%%%%%%
\newcommand{\Acal}{{\mathcal A}}
\newcommand{\Bcal}{{\mathcal B}}
\newcommand{\Ccal}{{\mathcal C}}
\newcommand{\Dcal}{{\mathcal D}}
\newcommand{\Ecal}{{\mathcal E}}
\newcommand{\Fcal}{{\mathcal F}}
\newcommand{\Gcal}{{\mathcal G}}
\newcommand{\Hcal}{{\mathcal H}}
\newcommand{\Ical}{{\mathcal I}}
\newcommand{\Jcal}{{\mathcal J}}
\newcommand{\Kcal}{{\mathcal K}}
\newcommand{\Lcal}{{\mathcal L}}
\newcommand{\Mcal}{{\mathcal M}}
\newcommand{\Ncal}{{\mathcal N}}
\newcommand{\Ocal}{{\mathcal O}}
\newcommand{\Pcal}{{\mathcal P}}
\newcommand{\Qcal}{{\mathcal Q}}
\newcommand{\Rcal}{{\mathcal R}}
\newcommand{\Scal}{{\mathcal S}}
\newcommand{\Tcal}{{\mathcal T}}
\newcommand{\Ucal}{{\mathcal U}}
\newcommand{\Vcal}{{\mathcal V}}
\newcommand{\Wcal}{{\mathcal W}}
\newcommand{\Xcal}{{\mathcal X}}
\newcommand{\Ycal}{{\mathcal Y}}
\newcommand{\Zcal}{{\mathcal Z}}

%%%%%%%%%%%%%%%%%%%%%%%%%%%%
% Blackboard Bold Alphabet %
%%%%%%%%%%%%%%%%%%%%%%%%%%%%
\renewcommand{\AA}{\mathbb{A}}
\newcommand{\BB}{\mathbb{B}}
\newcommand{\CC}{\mathbb{C}}
\newcommand{\FF}{\mathbb{F}}
\newcommand{\G}{\mathbb{G}}
\newcommand{\KK}{\mathbb{K}}
\newcommand{\NN}{\mathbb{N}}
\newcommand{\PP}{\mathbb{P}}
\newcommand{\QQ}{\mathbb{Q}}
\newcommand{\RR}{\mathbb{R}}
\newcommand{\ZZ}{\mathbb{Z}}

%%%%%%%%%%%%%%%%%%%%%%%%%%
% Boldface Math Alphabet %
%%%%%%%%%%%%%%%%%%%%%%%%%%
\newcommand{\bfa}{{\boldsymbol a}}
\newcommand{\bfb}{{\boldsymbol b}}
\newcommand{\bfc}{{\boldsymbol c}}
\newcommand{\bfd}{{\boldsymbol d}}
\newcommand{\bfe}{{\boldsymbol e}}
\newcommand{\bff}{{\boldsymbol f}}
\newcommand{\bfg}{{\boldsymbol g}}
\newcommand{\bfi}{{\boldsymbol i}}
\newcommand{\bfj}{{\boldsymbol j}}
\newcommand{\bfk}{{\boldsymbol k}}
\newcommand{\bfm}{{\boldsymbol m}}
\newcommand{\bfp}{{\boldsymbol p}}
\newcommand{\bfr}{{\boldsymbol r}}
\newcommand{\bfs}{{\boldsymbol s}}
\newcommand{\bft}{{\boldsymbol t}}
\newcommand{\bfu}{{\boldsymbol u}}
\newcommand{\bfv}{{\boldsymbol v}}
\newcommand{\bfw}{{\boldsymbol w}}
\newcommand{\bfx}{{\boldsymbol x}}
\newcommand{\bfy}{{\boldsymbol y}}
\newcommand{\bfz}{{\boldsymbol z}}
\newcommand{\bfA}{{\boldsymbol A}}
\newcommand{\bfF}{{\boldsymbol F}}
\newcommand{\bfB}{{\boldsymbol B}}
\newcommand{\bfD}{{\boldsymbol D}}
\newcommand{\bfG}{{\boldsymbol G}}
\newcommand{\bfI}{{\boldsymbol I}}
\newcommand{\bfM}{{\boldsymbol M}}
\newcommand{\bfP}{{\boldsymbol P}}
\newcommand{\bfX}{{\boldsymbol X}}
\newcommand{\bfY}{{\boldsymbol Y}}
\newcommand{\bfzero}{{\boldsymbol{0}}}
\newcommand{\bfone}{{\boldsymbol{1}}}

%%%%%%%%%%%%%%%%%%%%%%%%%%%%%%
% Miscellaneous New Commands %
%%%%%%%%%%%%%%%%%%%%%%%%%%%%%%
\newcommand{\aff}{{\textup{aff}}}
\newcommand{\Aut}{\operatorname{Aut}}
\newcommand{\Berk}{{\textup{Berk}}}
\newcommand{\Birat}{\operatorname{Birat}}
\newcommand{\characteristic}{\operatorname{char}}
\newcommand{\codim}{\operatorname{codim}}
\newcommand{\Crit}{\operatorname{Crit}}
\newcommand{\critwt}{\operatorname{critwt}} % valency of a portrait
\newcommand{\cond}{\operatorname{cond}}
\newcommand{\Cycle}{\operatorname{Cycles}}
\newcommand{\diag}{\operatorname{diag}}
\newcommand{\diam}{\operatorname{diam}}
\newcommand{\Disc}{\operatorname{Disc}}
\newcommand{\Div}{\operatorname{Div}}
\newcommand{\Dom}{\operatorname{Dom}}
\newcommand{\End}{\operatorname{End}}
\newcommand{\ExtOrbit}{\mathcal{EO}} %% Extended orbit
\newcommand{\Fbar}{{\bar{F}}}
\newcommand{\Fix}{\operatorname{Fix}}
\newcommand{\FOD}{\operatorname{FOD}}
\newcommand{\FOM}{\operatorname{FOM}}
\newcommand{\Gal}{\operatorname{Gal}}
\newcommand{\genus}{\operatorname{genus}}
\newcommand{\grac}{\operatorname{grac}}
\newcommand{\GITQuot}{/\!/}
\newcommand{\GR}{\operatorname{\mathcal{G\!R}}}
\newcommand{\Hom}{\operatorname{Hom}}
\newcommand{\Index}{\operatorname{Index}}
\newcommand{\Image}{\operatorname{Image}}
\newcommand{\Isom}{\operatorname{Isom}}
\newcommand{\hhat}{{\hat h}}
\newcommand{\Ker}{{\operatorname{ker}}}
\newcommand{\Ksep}{K^{\textup{sep}}}  %% separable closure of K
\newcommand{\lcm}{{\operatorname{lcm}}}
\newcommand{\LCM}{{\operatorname{LCM}}}
\newcommand{\Lift}{\operatorname{Lift}}
\newcommand{\limstar}{\lim\nolimits^*}
\newcommand{\limstarn}{\lim_{\hidewidth n\to\infty\hidewidth}{\!}^*{\,}}
\newcommand{\llog}{\log\log}
\newcommand{\logplus}{\log^{\scriptscriptstyle+}}
\newcommand{\maxplus}{\operatornamewithlimits{\textup{max}^{\scriptscriptstyle+}}}
\newcommand{\MOD}[1]{~(\textup{mod}~#1)}
\newcommand{\Mor}{\operatorname{Mor}}
\newcommand{\Moduli}{\mathcal{M}}
\newcommand{\Norm}{{\operatorname{\mathsf{N}}}}
\newcommand{\notdivide}{\nmid}
\newcommand{\normalsubgroup}{\triangleleft}
\newcommand{\NS}{\operatorname{NS}}
\newcommand{\onto}{\twoheadrightarrow}
\newcommand{\ord}{\operatorname{ord}}
\newcommand{\Orbit}{\mathcal{O}}
\newcommand{\Per}{\operatorname{Per}}
\newcommand{\Perp}{\operatorname{Perp}}
\newcommand{\PrePer}{\operatorname{PrePer}}
\newcommand{\PGL}{\operatorname{PGL}}
\newcommand{\Pic}{\operatorname{Pic}}
\newcommand{\Prob}{\operatorname{Prob}}
\newcommand{\Proj}{\operatorname{Proj}}
\newcommand{\Qbar}{{\bar{\QQ}}}
\newcommand{\rank}{\operatorname{rank}}
\newcommand{\rad}{\operatorname{rad}}
\newcommand{\Rat}{\operatorname{Rat}}
\newcommand{\Res}{{\operatorname{Res}}}
\newcommand{\Resultant}{\operatorname{Res}}
\renewcommand{\setminus}{\smallsetminus}
\newcommand{\sgn}{\operatorname{sgn}}
\newcommand{\SL}{\operatorname{SL}}
\newcommand{\Span}{\operatorname{Span}}
\newcommand{\Spec}{\operatorname{Spec}}
\renewcommand{\ss}{{\textup{ss}}}
\newcommand{\stab}{{\textup{stab}}}
\newcommand{\Stab}{\operatorname{Stab}}
\newcommand{\Support}{\operatorname{Supp}}
\newcommand{\Sym}{\operatorname{Sym}}  %% Symmetric group
\newcommand{\tors}{{\textup{tors}}}
\newcommand{\Trace}{\operatorname{Trace}}
\newcommand{\trianglebin}{\mathbin{\triangle}} % symmetric set difference
\newcommand{\tr}{{\textup{tr}}} % for K/k trace
\newcommand{\UHP}{{\mathfrak{h}}}    % Upper half plane
\newcommand{\Wander}{\operatorname{Wander}}
\newcommand{\<}{\langle}
\renewcommand{\>}{\rangle}

\newcommand{\pmodintext}[1]{~\textup{(mod}~#1\textup{)}}
\newcommand{\ds}{\displaystyle}
\newcommand{\longhookrightarrow}{\lhook\joinrel\longrightarrow}
\newcommand{\longonto}{\relbar\joinrel\twoheadrightarrow}
\newcommand{\SmallMatrix}[1]{%
  \left(\begin{smallmatrix} #1 \end{smallmatrix}\right)}
  
  \def\({\left(}
\def\){\right)}

\hypersetup{breaklinks=true}
%\def\[{\left[}
%\def\]{\right]}
%\def\<{\langle}
%\def\>{\rangle}

  %%%%%%%%%%%%%%%%%%%%%%%%%%%%%%%%%%

\title[Odd graceful coloring of graphs]
{Upper bounds on the odd graceful chromatic number of graphs}
\author[M. Afifurrahman, F. F. Hadiputra]{Muhammad Afifurrahman, Fawwaz Fakhrurrozi Hadiputra}

\address[MA]{School of Mathematics and Statistics, University of New South Wales, Sydney NSW 2052, Australia}
\email{m.afifurrahman@unsw.edu.au}
\address[FFH]{School of Mathematics and Statistics, The University of Melbourne, Parkville, VIC 3010, Australia}
\email{fhadiputra@student.unimelb.edu.au}

\thanks{}
\begin{abstract}We obtain several new upper bounds of the odd graceful chromatic number of a graph $G$, which must be bipartite. Some of our bounds depend only on the number of the vertices of $G$ or the chromatic number of some graphs related to the bipartition of $G$. 
\end{abstract}

\maketitle

%%%%%%%%%%%%%%%%%%%%%%%%%%%%%%%%%%%%%%%%%%%%%%%%%%%%%%%%%%%%%%%%%%%%%%
\tableofcontents
\section{Introduction}\label{sec:intro} 

\subsection{Odd graceful coloring}

For a graph $G$, a graph labeling on $G$ is a map that assigns graph elements on $G$ to elements of another set, usually positive integers. There has been a plethora of research on various graph labelings defined on a graph $G$; a dynamic survey is available at~\cite{Gal}.

In this work, we consider a graph labeling problem on a simple graph $G$ (that is connected, unless stated otherwise) over the vertex set $V(G)$. In particular, we work on the problem on odd graceful colorings of a graph $G$. This labeling was first introduced by Suparta et. al.~\cite{SLHB} as a generalization of odd graceful labeling and graceful coloring of graphs, which in turns generalizes the notion of graceful labelings.

We first recall the definitions. For a graph $G$ and a positive integer $k$, consider a vertex labeling $
    \lambda \colon V(G) \to \Z.$ We extend the labeling $\lambda$ to an edge labeling $\lambda'\colon E(G)\to \Z$ by defining \begin{align*}
    \lambda'(uv)=|\lambda(u)-\lambda(v)|\:\text{if}\: uv\in E(G).
\end{align*} 

The main work of this paper considers the \textit{odd graceful coloring},
defined as follows from~\cite{SLHB}. First, we define a graceful $k$-coloring, first suggested by Chartrand and introduced at~\cite{BBELZ}, as a vertex labeling $\lambda \colon V(G) \to \{0,1,\dots, k\}$ such that $\lambda$ induces a vertex coloring on $G$ (that is, no two adjacent vertices are of the same color).

Next, we define a graceful $k$-coloring $\lambda$ on a graph $G$ as a \textit{$k$-odd graceful coloring} if the value of the induced edge labeling $\lambda'$ is always odd for any edge in $G$. Similar to the previous concept, the \textit{odd graceful chromatic number} of $G$, $\chi_{og}(G)$ is defined as the smallest $k$ such that $G$ has a $k$-odd graceful coloring. If no such $k$ exists, we denote $\chi_{og}(G)=\infty$.
 Note that a similar but stronger concept, odd-graceful total coloring, is introduced at~\cite{SSY}, where the colorings are considered over the vertices and edges of $G$.

In their paper, Suparta \textit{et al.}~\cite{SLHB} gave some lower bounds of $\chi_{og}(G)$ for general graphs $G$, and calculate its exact value for the families of cycles, paths, some families of caterpillars, some generalized star graphs, ladder and some prism graphs. In addition, they also proved~\cite[Theorem~2]{SLHB} that any non-bipartite graph does not admit an odd graceful coloring.

In this paper, we provide several new upper bounds of $\chi_{og}(G)$ for general graphs $G$, with main results being Theorems~\ref{thm:upper_bound} and~\ref{thm:main-upper-bound}. In particular, we prove that an odd graceful total coloring of $G$ always exists for any bipartite graph $G$, thereby complementing~\cite[Theorem~2]{SLHB}. In addition, we also compute the odd graceful chromatic number of complete bipartite graphs and near-complete bipartite graphs, as seen in Theorems~\ref{thm:complete_bipartite} and~\ref{thm:near-complete}.

\subsection{Related graph labelings}
This section concerns some family of graph labelings $G$ which are related to the odd graph labelings.

A labeling $\lambda$ is \textit{graceful} if the codomain of $\lambda$ is $\{0,1,\dots,|E(G)|\}$ and the range of the induced edge labeling $\lambda'$ is exactly $\{1,\dots,|E(G)|\}$. In particular, not all graph has a graceful labeling. There have been plethora of works in graceful labelings and their variations, as seen at~\cite[Chapters 1, 2 and 3]{Gal}. We will discuss some of them in this section.

The first variant, the \textit{gracefulness} of a graph, introduced at~\cite{CL} and denoted as $\grac G$, is defined as the smallest integer $k$ such that there exist an injective vertex labeling $\lambda \colon V(G) \to \{0,1,\dots, k\}$ such that the associated edge labeling $\lambda'$ is injective. This parameter can be seen as a generalisation for a graceful labeling that holds for all graphs.

Another variant, the \textit{odd graceful labeling}, was said to be first introduced by Gnanajothi~\cite{Gnana} in their PhD thesis. However, the authors have not been able to verify this claim independently. A labeling $\lambda$ is odd graceful if the codomains of $\lambda$ and $\lambda'$ are $\{0,1,\dots,2|E(G)|-1\}$ and $\{1,3,\dots,2|E(G)|-1\}$, respectively.  References of the state of the art for this labeling are available at~\cite[Section~3.6]{Gal}. Note that any graph with an odd graceful labeling must be bipartite, which is said to be sufficient by Gnanajothi~\cite{Gnana}.

The next variant, the \textit{graceful coloring}, was first suggested by Chartrand and introduced at~\cite{BBELZ}, combines the concepts of gracefulness of $G$ and proper coloring of a graph as follows. For a positive integer $k$, define a \textit{graceful $k$-coloring} as a vertex labeling $\lambda \colon V(G) \to \{0,1,\dots, k\}$ such that $\lambda$ induces a vertex coloring on $G$ (that is, no two adjacent vertices are of the same color) and $\lambda$ also induces an edge coloring $\lambda'$ on $G$, where $\lambda'(uv)=|\lambda(u)-\lambda(v)|$ for all $uv\in E(G)$. Then, define the \textit{graceful chromatic number} of $G$, $\chi_g(G)$, as the smallest $k$ such that $G$ has a graceful $k$-coloring.

With respect to the odd graceful labeling, we first note that the concept of graceful coloring combines the concepts of gracefulness of $G$ and proper coloring of a graph. Then, the concept of odd graceful coloring combines graceful coloring with odd graceful labeling.

\subsection{Overview of results and proofs}
The structure of the paper goes as follows. Section~\ref{sec:intro} provides background on the odd graceful labelings, related graphs labelings and notations. Section~\ref{sec:chromatic} is devoted for results in upper bounds of $\chi_{og}(G)$ with respect to the chromatic number of a related graph. Section~\ref{sec:near-complete} discuss a new upper bound of $\chi_{og}(G)$ with respect to the number of vertices $|V(G)|$. In addition, this section also discusses the odd graceful chromatic number of the families of complete and near-complete bipartite graphs.

We provide some preliminary results in Section~\ref{sec:pres}, which are used to prove the results of Sections~\ref{sec:chromatic} and~\ref{sec:near-complete} in Sections~\ref{sec:chromatic-p} and~\ref{sec:near-complete-p}, respectively.

Most of our upper bounds arguments rely on constructing odd graceful labelings of a bipartite graph $G$ to satisfy the upper bounds we want to show. In the case of Theorem~\ref{thm:D} and Theorem~\ref{thm:D2}, we apply the well-known Brooks' theorem (see~\cite{Brooks}) on chromatic numbers to further improve the related bound.

To prove lower bounds (and exact calculations) of the graphs in Theorems~\ref{thm:complete_bipartite}-\ref{thm:near-complete}, we use several other arguments and simplifications. The main simplification comes from Lemma~\ref{lem:subgraph}, which allows us to consider a particular class of graphs. We further use tool from additive combinatorics, where the related additive structure arises from the relation stated in Lemma~\ref{lem:abcgraceful}. For example, Lemma~\ref{lem:sumset} is used to enforce a lower bound on the labeling of a complete bipartite graph $K_{m,n}$. These results would then be used to give a better upper bound for $\chi_{og}(G)$, as seen at Theorem~\ref{thm:main-upper-bound}.

\subsection{Definitions and notations}\label{sec:def}
We denote $K_{n}$ as the complete graph of $n$ vertices, and $K_{m,n}$ as the complete bipartite graph which can be partitioned to sets of $m$ and $n$ vertices. A circulant graph $Ci_n(s_1,s_2,\ldots,s_k)$ is a graph obtained by a vertex set
\begin{align*}
    V(Ci_n(s_1,s_2,\ldots,s_k)) = \{u_i \mid i \in [1,n]\}
\end{align*}
and an edge set
\begin{align*}
    E(Ci_n(s_1,s_2,\ldots,s_k)) = \{u_iu_{i+s_j} \mid i \in [1,n], j \in [1,k]\}
\end{align*}
where the index $i$ taken modulo $n$. A special case of this family of graph is the Möbius ladder graph $M_{2n}$, with $M_{2n} \cong Ci_{2n}(1,n)$.

We also recall a graph operation $G^2$ on a graph $G$, defined in~\cite{Die} for example. The graph $G^2$ is defined as a graph that satisfies
\begin{align*}
    V(G^2) & \coloneqq V(G), \quad E(G^2) \coloneqq \{uv \mid u,v \in V(G), d(u,v) \le 2\}.
\end{align*}
In other words, $G^2$ is a graph obtained by adjoining vertices in $G$ whose distance are exactly two. For example, since every two vertices in the complete bipartite graph $K_{m,n}$ has distance at most two, we have $(K_{m,n})^2 = K_{m+n}$. With respect to this operation, we define $G^2[U]$ as the subgraph of $G^2$ induced by the vertices of $U$. %With respect to this operation, we define $G^2[U]$ as the graph resulting from the application of the $G^2$ operation on a graph $U$.

For a bipartite graph $G$ such that $V(G)$ can be partitioned to two sets of vertices $U$ and $W$, we denote these two sets as the \textit{bipartition} of the graph $G$.

For an additive group $\mathbf{Z}$ and subsets $A,B\subseteq \mathbf{Z}$, we define the sumsets $A+B$ as the set \begin{align*}
    A+B\coloneqq \{ a+b\colon a\in A,b\in B\}.
\end{align*} However, we use the notation $2\cdot A=\{2a \colon a\in A\}$.

We recall that $\chi(G)$ is the chromatic number of the graph $G$ and $\Delta(G)$ is the largest degree of the vertices of $G$.

\section{New results}

\subsection{Upper bounds for general graphs}\label{sec:chromatic}
We prove several upper bounds on the value of $\chi_{og}(G)$, for arbitrary graph $G$. It is known from~\cite[Corollary~2]{SLHB} that if $G$ is not bipartite, then $\chi_{og}(G)=\infty$. Therefore, we can let $G$ be a bipartite graph in this work.

We are now ready to state our main result of this section, as follows.
\begin{theorem}\label{thm:upper_bound}
Let $G$ be a bipartite graph with its bipartition $U$ and $W$. Then, 
\begin{align*}
    \chi_{og}(G) \le 2(\chi(G^2[U]) + \chi(G^2[W]) - 1).
\end{align*} 
\end{theorem}

We now discuss some corollaries of Theorem~\ref{thm:upper_bound}. First, this theorem implies the existence of an odd graceful coloring for any bipartite graph. This complements~\cite[Corollary~2]{SLHB}, which shows that any graph with an odd graceful coloring must be bipartite.
\begin{corollary}
A graph $G$ admits an odd graceful coloring if and only if $G$ is bipartite.
\end{corollary}
Furthermore, we can also describe a property of an odd graceful coloring $\lambda$ on $G$.  Since any edges in a graph $G$ labeled by $\lambda$ are adjacent to vertices labeled with different parities, we obtain the following result on $\lambda$. \begin{corollary}\label{cor:parity}
    Let $\lambda$ be an odd graceful coloring $\lambda$ on a bipartite graph $G$ with bipartitions $U$ and $W$. Then, exactly one of these is true: 
    \begin{itemize}
        \item  $\lambda(u)$ is odd and $\lambda(w)$ is even for all $u\in U,\:w\in W$, or
        \item  $\lambda(u)$ is even and $\lambda(w)$ is odd for all $u\in U,\:w\in W$.
    \end{itemize}
\end{corollary}

Of course, we may also use  Theorem~\ref{thm:upper_bound} for specific classes of graphs. For example, following results on ladder graphs~\cite[Theorem 10]{SLHB}, we may compute the odd graceful chromatic number of the class of Möbius ladder graphs $M_{2n}$, defined in Section~\ref{sec:def}. Note that this graph is bipartite if and only if $n$ is odd, therefore we only consider this case, as follows.

\begin{figure}[h!]
    \centering
    \includegraphics[scale=0.6]{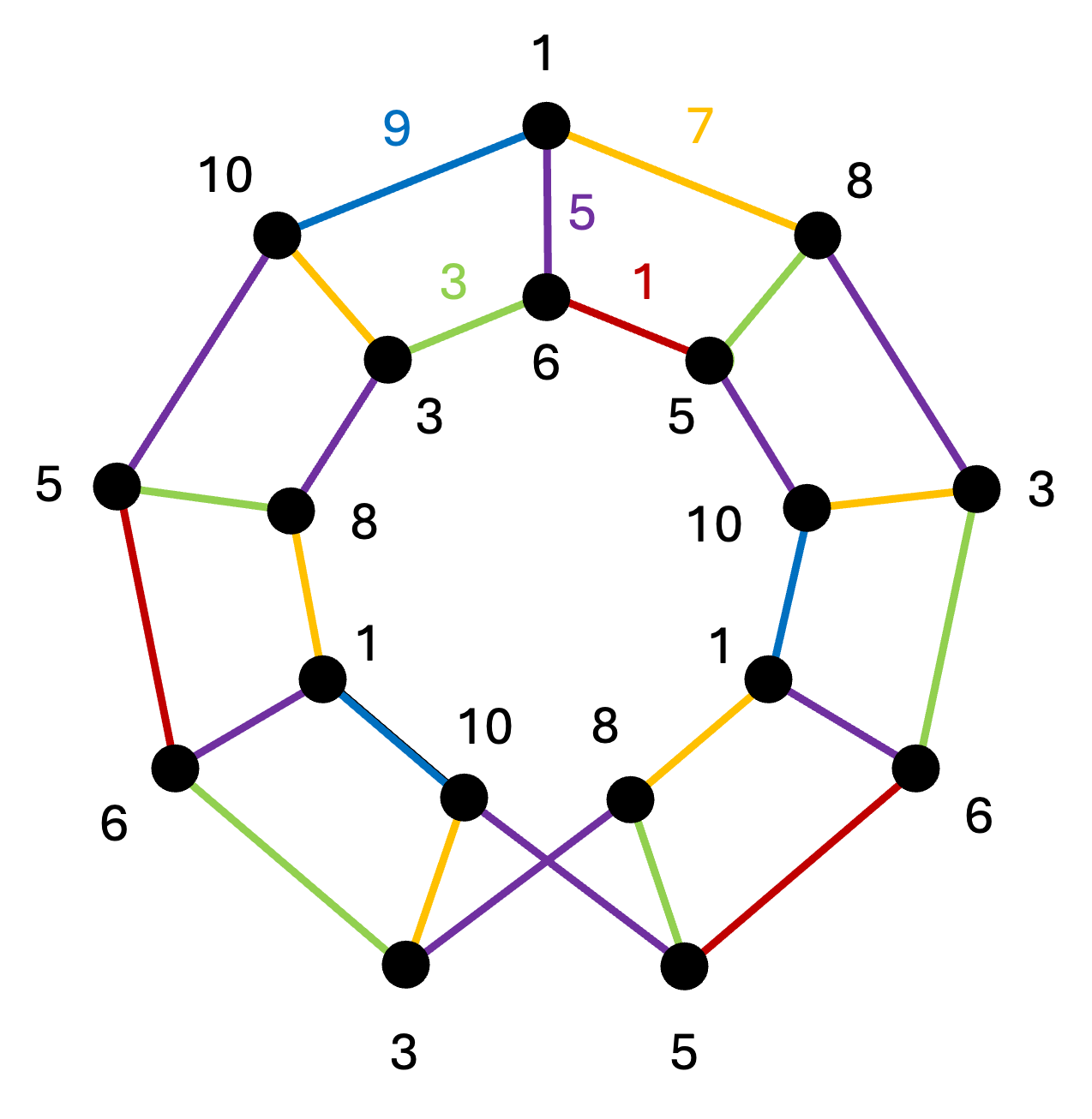}
    \caption{A 10-odd graceful coloring of $M_{18}$.}
\end{figure}

\begin{corollary}\label{cor:mob}
For odd $n \ge 3$, we obtain
\begin{align*}
    \chi_{og}(M_{2n}) & \le \begin{cases}
        10, & \text{ if } n \equiv 3 \pmod{6},\\
        14, & \text{ if } n \equiv 1,5 \pmod{6}, n > 5,\\
        18, & \text{ if } n = 5.
    \end{cases}
\end{align*}
\end{corollary}

 We now discuss applications of Theorem~\ref{thm:upper_bound} to more general classes of  bipartite graphs. For general graph, we first note that since $\chi(G^2[U])\le |V(G^2[U]|=|U|$ for any $U \subseteq V(G)$, we may obtain \begin{align}\label{eqn:vertices}
\chi_{og}(G)\le 2|V(G)|-2.    
\end{align} for any bipartite graph $G$. This bound will be further improved in Section~\ref{sec:near-complete}.

With Theorem~\ref{thm:upper_bound}, wee can also prove an upper bound that relates $\chi_{og}(G)$ with the maximum degree of $G$.

\begin{theorem}\label{thm:D}
Let $G$ be a bipartite graph with a bipartition $U$ and $W$.
 Then
\begin{align*}
    \chi_{og}(G) \le \begin{cases}
        4(\Delta(G))^2 - 4\Delta(G) - 2, & \text{if } G^2[U],\:G^{2}[W] \not\cong C_{2n+1} \: \text{and} \: K_{n+2}\\ & \text{ for some } n \in \mathbb{N},\\ 
        4(\Delta(G))^2 - 4\Delta(G) + 2, & \text{otherwise}.
    \end{cases}
\end{align*}
\end{theorem}
Note that Theorem~\ref{thm:D} is useful when $\Delta(G)\ge 3$. The case where G is a bipartite graph with $\Delta(G)\le 2$, where such graphs are either paths or cycles, has already been addressed in \cite{SLHB}. Related to this bound, we recall  a lower bound for $\chi_{og}$ from \cite[Lemma 3]{SLHB}, where for all (bipartite) graphs $G$, \[
\chi_{og}(G)\ge 2\Delta(G).
\] Some examples of the equality cases can be seen at~\cite{SLHB}.

On another hand, checking $G^2[U]$ and $G^2[W]$ is sometimes not trivial for a graph $G$. Therefore, we also obtain a related upper bound as a corollary of Theorem~\ref{thm:D} that does not require this condition.
\begin{proposition}\label{thm:D2}
Let $G$ be a bipartite graph and a bipartition $U$ and $W$ where $|U| \le |W|$. If $\diam(G) \ge 5, |E(G)| > 2|W|$ and $|U| \ge 4$, then $\chi_{og}(G) \le 4(\Delta(G))^2 - 4\Delta(G) - 2$.
\end{proposition}

In particular, any regular graph whose diameter is at least $5$ satisfy the condition given in Proposition~\ref{thm:D2}. As a sample of application, Proposition~\ref{thm:D2} can be applied directly to determine a bound on the odd graceful chromatic number of cubic graphs, as follows.
\begin{corollary}
For any cubic bipartite graph $G$ with $\diam(G) \ge 5$,
    $\chi_{og}(G) \le 22$.
\end{corollary}

\subsection{Complete and near-complete bipartite graphs}\label{sec:near-complete}In this section, we want to provide upper bounds for a general bipartite graph $G$ with respect to the number of vertices $|V(G)|$ is found at~\eqref{eqn:vertices}. We already have a non-trivial upper bound in~\eqref{eqn:vertices}, found by directly applying Theorem~\ref{thm:upper_bound}. We improve the bound in this section, culminating in the following result.

\begin{theorem}\label{thm:main-upper-bound}
Let $G$ be a bipartite, but not complete bipartite, graph. Then, $$\chi_{og}(G) \le 2|V(G)|-4.$$
\end{theorem}

Proving Theorem~\ref{thm:main-upper-bound} requires several new results on complete and near-complete bipartite graphs. First, we compute the odd graceful coloring number of complete bipartite graphs $K_{m,n}$, with $m\ge n\ge 2$.  
Note that in the case $n=1$, Suparta et. al. \cite{SLHB} proved \[\chi_{og}(K_{m,1})=2m-2. \]
\begin{theorem}\label{thm:complete_bipartite}
    For positive integers $m\ge n\ge 2$, we have\begin{align*}
        \chi_{og}(K_{m,n})=\begin{cases}
            2m+2n-3, &\text{if }(m,n)=(2s,2s)\text{ or }(2s,2),\:s\in \N,\\
            2m+2n-2, &\text{otherwise.}
        \end{cases}
    \end{align*}
\end{theorem}

\begin{figure}[h!]
    \centering
    \includegraphics[scale=0.6]{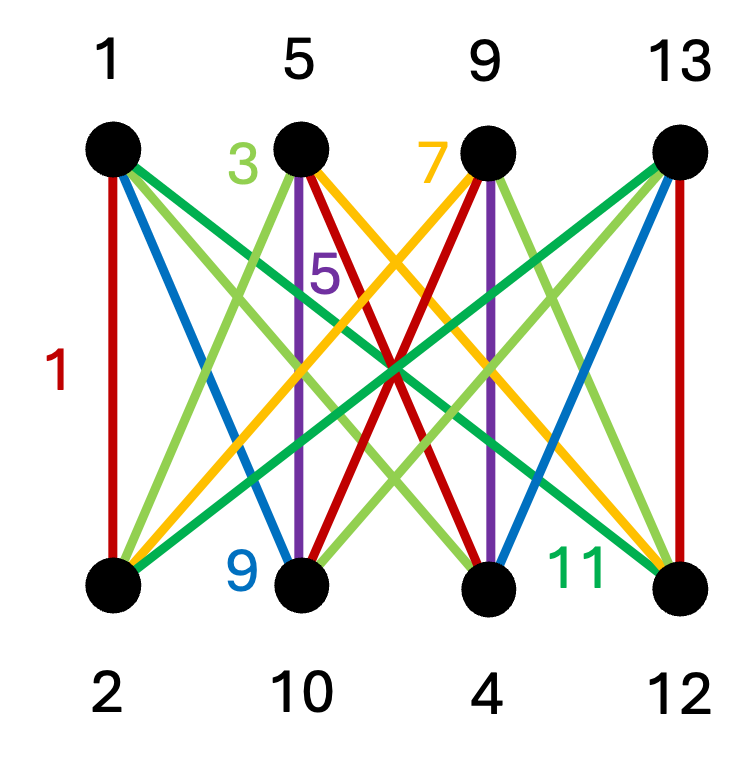}
    \caption{A 13-odd graceful coloring of $K_{4,4}$.}
\end{figure}

We further classify all $(2m+2n-3)$-odd graceful coloring of $K_{m,n}$.
\begin{theorem}\label{thm:label_complete_bipartite} Let $m,n\ge 2$.
    All possible odd graceful colorings $\lambda$ of $K_{m,n}=U \cup W$, with $|U|=m$, $|W|=n$, that realise the bound $\chi_{og}(K_{m,n})=2m+2n-3$ can be described as follows: \begin{itemize}
        \item When $(m,n)=(2s,2)$, \begin{align*}
        &\lambda(U)=\{1,4s+1\},\\
        &\lambda(W)=\{2i\colon 1\le i \le 2s\},
        \end{align*}
        \item When $(m,n)=(2s,2s)$, \begin{align*}
        &\lambda(U)=\{4i-3\colon 1\le i \le 2s\},\\
        &\lambda(W)=\{8j-6,8j-4\colon 1\le j \le s\}.
        \end{align*}
        (up to permutations between $U$ and $W$).
    \end{itemize}
\end{theorem}

The next step of proving Theorem~\ref{thm:main-upper-bound} is to compute the odd gracful chromatic number of near-complete bipartite graph, that is, graph that is constructed by removing some small number edges from a complete bipartite graph. In this work, we focus on computing the odd graceful chromatic number for the family of graphs $$K_{m,n}-K_{1,r}$$ with $m,n\ge 2, 1\le r$. This graph is defined  as the graph obtained by erasing $r$ incident edges from a vertices from the bipartition of $K_{m,n}$ with size $m$. Note that in this notation, $(K_{m,n}-K_{1,r})\not \cong (K_{n,m}-K_{1,r})$ unless $m=n$ or $r=1$.

\begin{theorem}\label{thm:near-complete}
    Let $m,n \ge 2$ and $r\le n$ be positive integers. Then, we have 
    \begin{equation*}\chi_{og}(K_{m,n}-K_{1,r})=2m+2n-5
    \end{equation*} if either \begin{itemize}
        \item $(m,n)=(3,2s)$ and $2\le r\le m$,
        \item $(m,n)=(2s+1,2s)$ and $s\le r\le 2s-1$, or
        \item $(m,n)=(2s+1,2s)$, $r=s-1$ and $s\ge 5$ is odd.
    \end{itemize}
    Otherwise, \begin{equation*}
        \chi_{og}(K_{m,n}-K_{1,r})=2m+2n-4.
    \end{equation*}
\end{theorem}

Based on these results, we are now ready to prove Theorem~\ref{thm:main-upper-bound}
\begin{proof}[Proof of Theorem~\ref{thm:main-upper-bound}]Note that any bipartite $G$ that are not complete bipartite graph is a subgraph of a graph obtained by removing an edge from a complete bipartite graph $K_{m,n}$, for some $m\ge n\ge 2$, which is $K_{m,n}-K_{1,1}$.

Hence, by applying Lemma~\ref{lem:subgraph} and taking $r=1$ in Theorem~\ref{thm:near-complete}, we obtain \[
\chi(G)\le \chi(K_{m,n}-K_{1,1})= 2m+2n-4 = 2|V(G)|-4,
\] which completes the proof.
\end{proof}
\section{Preliminary results}\label{sec:pres}
We first recall some previous results on graceful coloring~\cite{BBELZ} and odd graceful coloring~\cite{SLHB}. Note that any odd graceful coloring is also a graceful coloring, hence some results from~\cite{BBELZ} can be readily used in our setup.

\begin{lemma}\label{lem:abcgraceful}\cite[Observation~2.4]{BBELZ}
    Let $\lambda$ be a(n odd) graceful coloring of $G$. Then, for any path $(a,b,c)$ of length three, $2\lambda(b)\neq \lambda(a)+\lambda(c)$.
\end{lemma}

\begin{lemma}\label{lem:subgraph}\cite[Lemma~2]{SLHB}
Let $H$ be a subgraph of $G$. We have $\chi_{og}(H) \le \chi_{og}(G)$.    
\end{lemma}

Next, we also need a classical result on the sumsets of positive integers, as seen at~\cite{TV}.
\begin{lemma}\label{lem:sumset}\cite[Lemma~5.3 and Proposition~5.8]{TV} For any finite subset $A,B\subseteq \Z$, we have \begin{align*}
    |A+B|\ge |A|+|B|-1.
\end{align*} Moreover, equality holds if and only if $A$ and $B$ are arithmetic progression with same differences.
\end{lemma}
Another classical result used is Brooks' theorem~\cite{Brooks} on the chromatic number $\chi$ of a graph $G$.
\begin{proposition}\label{prp:brooks}
For any graph $G$ we have \begin{align*}
    \chi(G)\le \begin{cases}
        \Delta(G), & \text{if } G \not\cong C_{2n+1} \text{ and } G \not\cong K_{n+2} \text{ for } n \in \mathbb{N},\\
        \Delta(G)+1, & \text{otherwise.}
    \end{cases} 
\end{align*}The inequality is an equality when $G\cong C_{2n+1}$ or $K_{n+2}$.
\end{proposition}

\section{Upper bounds with respect to some chromatic numbers}\label{sec:chromatic-p}
\subsection{Proof of Theorem~\ref{thm:upper_bound}}
We prove that \begin{align*}
    \chi_{og}(G) \le 2(\chi(G^2[U]) + \chi(G^2[W]) - 1),
\end{align*} for any bipartite graph $G$ with bipartition $U$ and $W$.

Let $k = 2(\chi(G^2[U]) + \chi(G^2[W]) - 1)$. Let $V(G) = U \cup W$ be a bipartition of $G$ such that $U = \{u_1,\ldots,u_m\}$ and $W = \{w_1,\ldots,w_n\}$ for some positive integer $m$ and $n$. Let $\psi_1 : V(G^2[U]) \to [1,\chi(G^2[U])]$ and $\psi_2 : V(G^2[W]) \to [1,\chi(G^2[W])]$ be a vertex coloring of $G^2[U]$ and $G^2[W]$ respectively. Define a vertex labeling $\varphi : V(G) \to [1,k]$ where
\begin{align*}
    \varphi(u_i) & = 2 \cdot \psi_1(u_i) - 1,\\
    \varphi(w_j) & = 2(\psi_2(w_j) + \chi(G^2[U]) - 1).
\end{align*}
We now prove that $\varphi$ is an odd graceful coloring of $G$.

Observe that $\varphi(x) \ne \varphi(y)$ for any $xy \in E(G^2)$. Since $\varphi(u_i) < \varphi(w_j)$ for any $i$ and $j$, we obtain $|\varphi(u_i) - \varphi(w_j)| = |\varphi(u_\ell) - \varphi(w_j)|$ if and only if $\varphi(u_i) = \varphi(u_\ell)$ for any $l$. Similarly, $|\varphi(u_i) - \varphi(w_j)| = |\varphi(u_i) - \varphi(w_\ell)|$ if and only if $\varphi(w_j) = \varphi(w_\ell)$ for any $l$. Consider the edges $xy$ and $yz$ in $G$. Then $d(x,z) = 2$ which implies $xz \in E(G^2)$. Hence, $\varphi(x) \ne \varphi(z)$ which implies $|\varphi(x) - \varphi(y)| \ne |\varphi(z) - \varphi(y)|$.

This shows that $\varphi$ is an odd graceful coloring of $G$, with $\max(\varphi) = k$, which completes the proof.

\subsection{Proof of Corollary~\ref{cor:mob}}

Let $U$ and $W$ be the bipartitions of $M_{2n}$. Observe that 
\begin{align*}
    (M_{2n})^2[U] \cong Ci_n(1,2) \cong (M_{2n})^2[W].
\end{align*}

Now, let us consider the chromatic numbers of $Ci_n(1,2)$ when $n$ is odd. If $n = 5$, then $Ci_n(1,2) \cong K_5$ which implies $\chi(Ci_5(1,2)) = 5$. By Theorem \ref{thm:upper_bound}, it follows that
\begin{align*}
    \chi_{og}(M_{10}) \le 2(5 + 5 - 1) = 18.
\end{align*}

Next, consider $n \equiv 3 \pmod 6$. Define a map $f : V(Ci_n(1,2)) \to \{1,2,3\}$ given by $f(u_i) = i \pmod 3$. This is a proper coloring of $Ci_n(1,2)$ since any two vertices with the same color $u_i$ and $u_j$ implies $i \equiv j \pmod 3$ which yields $|i-j| \notin \{1,2\}$ or equivalently $u_i$ and $u_j$ are not adjacent. Therefore, $\chi(Ci_n(1,2)) \le 3$ which implies
\begin{align*}
    \chi_{og}(M_{2n}) \le 2(3 + 3 - 1) = 10
\end{align*}
due to Theorem \ref{thm:upper_bound}.

Lastly, let $n \equiv 1,5 \pmod 6$ where $n > 5$.
Consider a map $h : V(Ci_n(1,2)) \to \{1,2,3,4\}$ defined by
\begin{align*}
    h(u_i) = \begin{cases}
        1, & \text{if } i = 3,\\
        2, & \text{if } i = 4,\\
        3, & \text{if } i = 2,\\
        4, & \text{if } i \in \{1,5\},\\
        f(u_i), & \text{otherwise}.
    \end{cases}
\end{align*}
We will show that $h$ is a proper vertex coloring of $Ci_n(1,2)$. First, observe that the only two vertices of color $4$ are $u_1$ and $u_5$ where these two are not adjacent. Next, note that $h(u_6) = 3$ and $h(u_n) \in \{1,2\}$. Therefore, it is easy to see that any two adjacent vertices $u_i$ and $u_j$ where $i,j \in [1,6]\cup\{n\}$ have distinct colors. Furthermore, if $u_i$ and $u_j$ have the same color and $i,j \notin [1,5]$, then $u_i$ and $u_j$ are not adjacent since $f$ is a proper vertex coloring. This shows that $h$ is a proper coloring of $Ci_n(1,2)$ which implies $\chi(Ci_n(1,2)) \le 4$. Again, by Theorem \ref{thm:upper_bound}, it holds that
\begin{align*}
    \chi_{og}(M_{2n}) \le 2(4 + 4 - 1) = 14,
\end{align*}
which completes the proof.

\subsection{Proof of Theorem~\ref{thm:D}}

Let $v \in W$ be fixed. Then $\deg_G(v) \le \Delta(G)$. Each neighbors of $v$ is also adjacent to at most $\Delta(G) - 1$ vertices other than $v$. It follows that
\begin{align*}
    |\{v' \mid v' \in W, d(v',v) = 2\}| \le \Delta(G)(\Delta(G)-1).
\end{align*}
Since $v$ is chosen at random, therefore $\Delta(G^2[W]) \le \Delta(G)(\Delta(G) - 1)$. Using similar argument, we also have $\Delta(G^2[U]) \le \Delta(G)(\Delta(G)-1)$.

Applying Brooks' Theorem in Proposition~\ref{prp:brooks}, it follows that
\begin{align*}
    \chi_{og}(G) & \le 2(\chi(G^2[U]) + \chi(G^2[W]) - 1),\\
    & \le 2(\Delta(G^2[U]) + 1 + \Delta(G^2[W]) + 1 - 1),\\
    & \le 2(2\Delta(G)(\Delta(G)-1) + 1),\\
    & \le 4(\Delta(G))^2 - 4\Delta(G) +2.
\end{align*}

Now, let $G^2[U], G^2[W] \not\cong C_{2n+1}$ and $G^2[U], G^2[W] \not\cong K_{n+2}$ for $n \in \mathbb{N}$. Then $\chi(G^2[U]) \le \Delta(G)$ and $\chi(G^2[W]) \le \Delta(G)$ due to Brook's Theorem. In this case, it holds that
\begin{align*}
    \chi_{og}(G) & \le 2(\chi(G^2[U]) + \chi(G^2[W]) - 1),\\
    & \le 2(\Delta(G^2[U]) + \Delta(G^2[W]) - 1),\\
    & \le 2(2\Delta(G)(\Delta(G)-1) - 1),\\
    & \le 4(\Delta(G))^2 - 4\Delta(G) -2,
\end{align*}
which completes the proof.

\subsection{Proof of Proposition~\ref{thm:D2}}
Let $\diam(G) \ge 5$. Assume that for every pair of vertices $x,y \in U$ it holds that $d(x,y) = 2$. Fix any two vertices $w,z \in W$. Let $w' \in N(w)$ and $z' \in N(z)$. Since $\{w',z'\} \subseteq U$, then there exists $t \in W$ such that $t \in N(w') \cap N(z')$. It follows that
\begin{align*}
    w \sim w' \sim t \sim z' \sim z
\end{align*}
which shows that $d(w,z) \le 4$. Since $w$ and $z$ are chosen at random, this implies that $\diam(G) \le 4$ which is a contradiction. Therefore, there will always exists two vertices in $U$ with the distance larger than 2. This implies $G^2[U]$ is not isomorphic to a complete graph. Using similar argument, $G^2[W]$ is also not isomorphic to a complete graph.

Further, if $|E(G)| > 2|W|$ then by pigeonhole principle there exists $x \in W$ such that $\deg_G(x) \ge 3$. Hence, at least three neighbors of $x$ in $U$ forms a triangle in $G^2[U]$. Since $|U| \ge 4$, then $G^2[U]$ is not isomorphic to a cycle. Since $|E(G)| > 2|W| \ge 2|U|$, then $G^2[W]$ is also not isomorphic to a cycle.

This shows that $G^2[U], G^2[W] \not\cong C_{2n+1}$ and $G^2[U], G^2[W] \not\cong K_{n+2}$ for $n \in \mathbb{N}$. It follows that $\chi_{og}(G) \le 4(\Delta(G))^2 - 4\Delta(G) - 2$ by Theorem~\ref{thm:D} which completes the proof.

\section{Exact value of some odd coloring chromatic numbers}\label{sec:near-complete-p}

\subsection{Proof of Theorem~\ref{thm:complete_bipartite}}

We first prove that \begin{align}\label{eqn:lower_bound_2m+2n-3}
    \chi_{og}(K_{m,n})\ge 2m+2n-3.
\end{align}Without loss of generality, suppose $m\ge n$.
    Let  $\lambda$ be a $k$-odd graceful coloring of $K_{m,n}$. For $i=1,2,3,4$ define \begin{align*}
    \begin{split}
        T_{i}\coloneqq \{t \in V(G) \colon \lambda(t)\equiv i\pmod 4\}.
    \end{split}
    \end{align*} We also define $\lambda(T_i)\coloneqq \{ \lambda(t_i)\colon t_i\in T_i\}$. Trivially, \begin{align*}
       k=\max_{i=1,2,3,4}(\lambda(T_i)).
    \end{align*}
    
   We may see that $T_1$, $T_2$, $T_3$, $T_4$ partition $V(G)$. Furthermore, by using Corollary~\ref{cor:parity}, the sets $T_1 \cup T_3=U$ and $T_2 \cup T_4=W$ is a bipartition of $K_{m,n}$. Note that we have either \begin{align}\label{eqn:V1V2}
       (|U|,|W|)=(m,n) \quad \text{or}\quad (n,m).
   \end{align}
   Also, we note that at least one of $T_1$ and $T_3$, and also $T_2$ and $T_4$ are nonempty.

   We now divide the cases based on the number of nonempty sets from $T_1$, $T_2$, $T_3$ and $T_4$. From the last observation, at most two of them are empty. As the first (and worst) case, suppose that $T_3$ and $T_4$ are empty. Note that $T_1=U$ and $T_2=W$. From~\eqref{eqn:V1V2}, we note that either $|U|=m$ or $|W|=m$. We consider each case separately. First, suppose that $|U|=m$. We note that all elements in $U=T_1$ is $1$ modulo $4$. Since there are $m$ different numbers in $T_1$, in this case we have \begin{align*}
     k\ge  \max(\lambda(u)\colon u\in U) \ge 4m-3 \ge 2(m+n)-3.
   \end{align*} If $|W|=m$, with similar arguments we obtain \begin{align*}
       k\ge \max(\lambda(w)\colon w\in W) \ge 4m-2 \ge 2(m+n)-2.
   \end{align*} In either case,~\eqref{eqn:lower_bound_2m+2n-3} holds true, which completes the proof when $T_3$ and $T_4$ are empty.

   For the rest of the case where exactly two sets are empty, we note that if $T_1$ is empty (and $U=T_3$) we obtain $k\ge 4m-3 \ge 2(m+n)-1$. Also, if $T_2$ is empty, we obtain $k\ge 4m\ge 2(m+n)$. 
   
   Hence, if two of these four sets are empty, \begin{align}\label{eqn:lower_bound_4m-3}
       k\ge 4m-3 \ge 2(m+n)-3,
   \end{align}
   which proves~\eqref{eqn:lower_bound_2m+2n-3}.
   
   Next, we proceed to the case where at most one of the sets $T_1$, $T_2$, $T_3$ and $T_4$ are nonempty. Note that this case implies that either both $T_1$ and $T_3$ are nonempty, or both $T_2$ and $T_4$ are nonempty. We now consider this case according to these subcases instead.

     First, suppose that $T_1$ and $T_3$ are nonempty.  From Lemma~\ref{lem:abcgraceful}, we have that if $\lambda$ is a graceful coloring, for any $t_1\in T_1$, $t_3\in T_3$ and $w\in W$, \[
    \lambda(t_1)+\lambda(t_3)\ne 2 \lambda(w).
    \] This implies that the set $\lambda(T_1)+\lambda(T_3)$ and $2\cdot \lambda(W)$ are disjoint. 
    Therefore, we have the inequality \begin{align}\label{eqn:U1U3}\begin{split}
|\lambda(T_1)+\lambda(T_3)|+|2\cdot\lambda(W)| &\ge |\lambda(T_1)|+(|U|-|\lambda(T_1)|)-1+|W|\\&=m+n-1,
    \end{split}\end{align}
        where the first inequality comes from Lemma~\ref{lem:sumset}. This  implies \begin{align*}
            |[\lambda(T_1)+\lambda(T_3)] \cup2\cdot\lambda(W)| \ge m+n-1.
        \end{align*}
Next, note that all numbers that are in one of theses sets are positive numbers that are divisible by 4. Therefore, \begin{align}\label{eqn:maxU1U3}
    \max([\lambda(T_1)+\lambda(T_3)] \cup2\cdot\lambda(W)) \ge 4(m+n-1).
\end{align}
From the last inequality, we consider two separate cases: \begin{itemize}
    \item If $\max 2 \cdot\lambda(W) \ge 4(m+n-1)$, then $\max \lambda(W) \ge 2m+2n-2 $. This implies there exists a vertex $w$ with $\lambda(w)\ge 2m+2n-2$, which in turn implies $k \ge 2m+2n-2$.
    \item If $\max [\lambda(T_1)+\lambda(T_3)] \ge 4(m+n-1)$, then at least one of $\lambda(T_1)$ and $\lambda(T_3)$ has an element whose value is at least $2(m+n-1)$. This also implies $k\ge 2m+2n-2$.
\end{itemize}
Therefore, in either case we have $k\ge 2m+2n-2$, which completes the proof in the case where $T_1$ and $T_3$ are nonempty.

Next, suppose that $T_2$ and $T_4$ are both nonempty. We may apply the same argument as in the inequality~\eqref{eqn:U1U3} to the sets $T_2$, $T_4$  and $U$ to obtain \begin{align*}
\begin{split}|\lambda(T_2)+\lambda(T_4)|+|2\cdot\lambda(U)| &\geq |\lambda(T_2)|+(|U|-|\lambda(T_2)|)-1+|U|\\
    &= m+n-1.
\end{split}
\end{align*}
Since the numbers in these sets are $2$ modulo $4$, this implies an inequality similar to~\eqref{eqn:maxU1U3}, \begin{align}\label{eqn:maxU2U4}
    \max([\lambda(T_2)+\lambda(T_4)] \cup2\cdot\lambda(U)) \ge 4(m+n-1)-2.
\end{align} By applying a similar argument, we also have \begin{align*}
   k= \max(\lambda(T_2),\lambda(T_4),\lambda(U))\ge 2m+2n-3
\end{align*} if both of $T_2$ and $T_4$ are nonempty, which completes the proof of~\eqref{eqn:lower_bound_2m+2n-3}.

Next, from Theorem~\ref{eqn:vertices} and~\eqref{eqn:lower_bound_2m+2n-3}, we have \begin{equation*}
    \chi(K_{m,n})\in \{2m+2n-3,2m+2n-2\}.
\end{equation*}
We now classify all complete bipartite graph $K_{m,n}$ such that $\chi_{og}(K_{m,n})=2m+2n-3$. We stil let $m\ge n$.

Let $\lambda$ be a $k$-odd graceful labeling and consider the corresponding vertices set $T_1$, $T_2$, $T_3$, $T_4$. Note that from~\eqref{eqn:lower_bound_4m-3}, if two of these sets are empty, \begin{align*}
    k\ge 4m-3>2(m+n)-3.
\end{align*} Therefore, any $(2m+2n-3)$-odd graceful labeling of $K_{m,n}$ is not of this form.

Next, from the previous observation, if both $T_1$ and $T_3$ are nonempty, we have \begin{align*}
    k\ge 2m+2n-2.
\end{align*} Therefore, any $(2m+2n-3)$-odd graceful labeling for $K_{m,n}$ is also not of this form.

Now, let $\lambda$ be a $(2m+2n-3)$-odd graceful labeling of $K_{m,n}$. From the previous observations, we have that $T_2$ and $T_4$ are nonempty, but one of $T_1$ and $T_3$ are empty. Note that the equality case in~\eqref{eqn:maxU2U4} is attained in this case. Recall that all elements in the corresponding set $[\lambda(T_2)+\lambda(T_4)] \cup2\cdot\lambda(U)$ are $2$ modulo $4$. Therefore, since equality is attained, \begin{align}\label{eqn:explicit}
    [\lambda(T_2)+\lambda(T_4)] \cup2\cdot\lambda(U) = \{4i-2\: \colon 1\le i \le m+n-1\}.
\end{align}
Next, we note that $2\notin \lambda(T_2)+\lambda(T_4)$. Therefore, $2\in  2\cdot\lambda(U)$ and $1\in \lambda(U)$. This would imply $T_1$ is nonempty and $T_3$ is empty, and $U=T_1$. Therefore,  \begin{equation}\label{eqn:2|m+n}
    2m+2n-3 \in \lambda(T_1)\implies 2|m+n.
\end{equation} 
Also, since all elements of $U$ are $1$ modulo $4$,  \begin{align*}
    |U| \le (m+n)/2 \le m \implies |U|=n,\: |W|=m.
\end{align*}
We now determine $\lambda(T_2)$ and $\lambda(T_4)$ explicitly. First, since $T_3$ is empty, $6\notin 2\lambda(T_1)=2\cdot\lambda(U)$, which implies $6\in \lambda(T_2)+\lambda(T_4)$. This easily implies \begin{align*}
    2\in \lambda(T_2),\quad 4\in \lambda(T_4).
\end{align*}

Next, we consider the second-largest element of~\eqref{eqn:explicit}, $4(m+n-2)-2=4(m+n)-10$. Not that this element is $6$ modulo $8$, since $4|2(m+n)$. Therefore, $4(m+n)-10\notin 2\cdot\lambda(U)$, which implies $4m+4n-10 \in \lambda(T_2)+\lambda(T_4)$. However, note that the largest possible elements of $\lambda(T_2)$ and $\lambda(T_4)$ are $2m+2n-6$ and $2m+2n-4$, respectively. From these, we in fact have\begin{equation}\label{eqn:2m+2n-6,4}
    2m+2n-6 \in \lambda(T_2),\quad 2m+2n-4\in \lambda(T_4).
\end{equation}

Now, from Lemma~\ref{lem:sumset}, since the equality is attained at Equation~\ref{eqn:maxU2U4}, we have that $\lambda(T_2)$ and $\lambda(T_4)$ are two arithmetic progressions with same differences $d$. Counting the cardinalities of these sets, we obtain\begin{align}\label{eqn:d}\begin{split}
    &\dfrac{(2m+2n-6-2)}{d} +1 + \dfrac{(2m+2n-4-4)}{d} +1 =  m \\ \iff  &d= \dfrac{4m+4n-16}{m-2}.
\end{split}
\end{align} 
Since $m\ge n$, we have \begin{align*}
    d\le \dfrac{8m-16}{m-2}=8.
\end{align*} Note that all elements in $\lambda(T_2)$ are $2$ modulo $4$, thus $4|d$. Therefore, \begin{align*}
    d\in \{4,8\}.
\end{align*}

When $d=4$, we obtain, from~\eqref{eqn:d}, \begin{align*}
    4m-8=4m+4n-16 \iff n=2.
\end{align*} In this case, we also have $m$ is even from~\eqref{eqn:2|m+n}.

Next, when $d=8$, we obtain, from~\eqref{eqn:d}, \begin{align*}
    8m-16=4m+4n-16 \iff m=n.
\end{align*}
In addition, all elements in $\lambda(T_2)$ are congruent modulo $8$. Hence, \begin{align*}
    8|(2m+2n-6)-2 \iff 2|n.
\end{align*}

Therefore, we obtain \begin{align}\label{eqn:condition_2m+2n-3}
    \chi_{og}(K_{m,n})=2m+2n-3 \implies (m,n)=(2s,2s)\text{ or }(2s,2),\:s\in \N,
\end{align}
This completes the proof of the second equation in Theorem~\ref{thm:complete_bipartite}. Next, we proceed to show the converse of~\eqref{eqn:condition_2m+2n-3} by constructing the corresponding odd graceful labeling.

First, when $(m,n)=(2s,2)$, we consider the graph labeling $\varphi\colon V(K_{m,2})\to [1,2m+1] $ defined as \begin{equation}\label{eqn:labeling-2s,2}
    \varphi(u_i)=2i,\:1\le i\le m,\quad \varphi(w_1)=1,\quad \varphi(w_2)=2m+1.
\end{equation}
This labeling can be checked to be an odd graceful labeling, which completes the proof of Theorem~\ref{thm:complete_bipartite} in this case.

Next, when $m=n=2s$ for some positive integer $s$, we consider the graph labeling $\varphi\colon V(K_{2s,2s})\to [1,8s-3] $ defined as \begin{align}\label{eqn:labeling-2s,2s}\begin{split}
    &\varphi(u_i)=\begin{cases}
        8i-6,\: &1\le i \le s,\\
        8(i-s)-4,\: &s+1\le i\le 2s, 
    \end{cases}\\
    &\varphi(w_i)=4i-3,\:1\le i\le m.
\end{split}
\end{align}
We now prove that this labeling is a $(2m+2n-3)$-odd labeling. Consider the vertices $w_i$ for $i \in [1,m]$. Suppose there exists $j,k \in [1,2s], j\ne k$ such that $|\varphi(w_i)-\varphi(u_j)| = |\varphi(w_i)-\varphi(u_k)|$. Since $\varphi(u_j) \ne \varphi(u_k)$, we have $2\varphi(w_i) = \varphi(u_j) + \varphi(u_k)$. 

We now divide the cases according to the size of $j$ and $k$. If $j \in [1,s]$ and $k \in [s+1,2s]$ then
\begin{align*}
    2\varphi(w_i) & = \varphi(u_j) + \varphi(u_k),\\
    2(4i - 3) & = 8j - 6 + 8(k-s) - 4,\\
    8(i-j-k+s) & = -4,
\end{align*}
which is impossible. If $j,k \in [1,s], j \ne k$ then
\begin{align*}
    2\varphi(w_i) & = \varphi(u_j) + \varphi(u_k),\\
    2(4i - 3) & = 8j - 6 + 8k - 6,\\
    8(i-j-k) & = -6,
\end{align*}
which is a contradiction. Otherwise, if $j,k \in [s+1,2s], j \ne k$ then
\begin{align*}
    2\varphi(w_i) & = \varphi(u_j) + \varphi(u_k),\\
    2(4i - 3) & = 8(j-s) - 4 + 8(k-s) - 4,\\
    8(i-j-k+2s) & = -2,
\end{align*}
which again is impossible. Therefore $|\varphi(w_i)-\varphi(u_j)| \ne |\varphi(w_i)-\varphi(u_k)|$ for any $j,k \in [1,2s], j \ne k$. Now, consider the vertices $u_i$ for $i \in [1,2s]$. Suppose there exists $j,k \in [1,m], j\ne k$ such that $|\varphi(u_i)-\varphi(w_j)| = |\varphi(u_i)-\varphi(w_k)|$. Since $\varphi(w_j) \ne \varphi(w_k)$, we have $2\varphi(u_i) = \varphi(w_j) + \varphi(w_k)$. 

We now divide the cases according to the size of $i$. If $i \in [1,s]$ then
\begin{align*}
    2\varphi(u_i) & = \varphi(w_j) + \varphi(w_k),\\
    2(8i-6) & = 4j - 3 + 4k - 3,\\
    4(4i-j-k) & = 6,
\end{align*}
which is a contradiction. Otherwise, if $i \in [s+1,2s]$ then
\begin{align*}
    2\varphi(u_i) & = \varphi(w_j) + \varphi(w_k),\\
    2(8(i-s)-4) & = 4j - 3 + 4k - 3,\\
    4(4i-4s-j-k) & = 2,
\end{align*}
which is impossible. This implies $|\varphi(u_i)-\varphi(w_j)| = |\varphi(u_i)-\varphi(w_k)|$ for any $j,k \in [1,m], j\ne k$. This shows that $\varphi$ is an odd graceful labeling of $K_{2s,2s}$.

\subsection{Proof of Theorem~\ref{thm:label_complete_bipartite}} 
First, let $m,n\in \N$ such that $\chi_{og}(K_{m,n})=2m+2n-3$. We recall that either $n=2$ and $m=2s$ or $m=n=2s$, for some $s\in \N$. 

We recall the arguments of the previous section, especially~\eqref{eqn:2m+2n-6,4} and \eqref{eqn:d}. First, suppose that $n=2$ and $m=2s$. Note that from~\eqref{eqn:explicit},~\eqref{eqn:2|m+n} and~\eqref{eqn:2m+2n-6,4}, \[
1, 4s+1 \in \lambda(T_1), \quad 4s-2 \in \lambda(T_2), \quad 4s \in \lambda(T_4).
\] In order to find all possible labeling, we first note that either $(|U|,|W|)=(m,2)$ or $(2,m)$.

First, assume that $|W|=2$. From~\eqref{eqn:d} and previous observations, this would imply $\lambda(U)=\lambda(T_1)$ forms an arithmetic progression of size $m=2s$, difference $4$, and largest element $4s+1$. This is impossible, hence $(|U|,|W|)=(2,m)$.

Returning to~\eqref{eqn:d} in this case, we note that $\lambda(T_2)$ and $\lambda(T_4)$ are each arithmetic progressions of difference $4$ and largest element $4s-2$ and $4s$, respectively. From these information, we obtain $|\lambda(T_2)|,|\lambda(T_4)|\le s$. Since $|W|=m=2s=|T_2|+|T_4|$, the last inequalities are, in fact, equalities. Furthermore, we can also obtain \begin{align*}
    &\lambda(T_2)=\{4i-2\colon 1\le i \le s\},\\
    &\lambda(T_4)=\{4i\colon 1\le i \le s\}.
\end{align*} Note that now we have all vertices of $K_{m,n}$ labeled. To complete the argument, we only need to check whether this labeling is an odd graceful labeling, which is done at~\eqref{eqn:labeling-2s,2}.

Next, we consider the case $m=n=2s$. From~\eqref{eqn:d}, we obtain $\lambda(T_2)$ and $\lambda(T_4)$ are arithmetic progressions of common difference $8$. Arguing similarly as in the previous part, we have \begin{align*}
    &\lambda(T_2)=\{8i-6\colon 1\le i \le s\},\\
    &\lambda(T_4)=\{8i-4\colon 1\le i \le s\}.
\end{align*} We also note that all elements in $\lambda(U)$ are $1$ mod 4, and lie between $1$ and $2m+2n-3=8s-3$. Since there are exactly $2s$ numbers that are $1$ modulo $4$ between these intervals, we obtain \[
\lambda(U)=\{4i-3\colon 1\le i \le 2s\}.
\] This labeling is an odd graceful labeling from~\eqref{eqn:labeling-2s,2s}, which completes our proof for this case.

\subsection{Proof of Theorem~\ref{thm:near-complete}}
\subsubsection{Removing edges from a complete bipartite graph}Suppose $m\ge n\ge 2$.  We first consider the case $r=1$, where we prove \begin{equation}\label{eqn:complete-minus-one}
        \chi_{og}(K_{m,n} -K_{1,1}) \le 2m+2n-4.
    \end{equation}

We first provide $(2m+2n-4)$-odd  graceful labeling $\varphi$ of $G=K_{m,n}-K_{1,1}$.
We  label the graph vertices as $U = \{u_1,u_2,\ldots,u_m\}$ and $W = \{w_1,w_2,\ldots,w_n\}$ such that $U$ and $W$ are bipartitions of $G$ and $u_iw_j \in E(G)$ for all $1\le i\le m,\: 1\le j\le n$, except for $u_mw_1$.
Now, let $\varphi : V(G) \to [1,2m+2n-4]$ be a vertex labeling with
\begin{align*}
    \varphi(u_i) & \coloneqq 2i - 1,\\
    \varphi(w_j) & \coloneqq 2(j + m - 2).
\end{align*}
We now show that $\varphi$ is an odd graceful coloring of $G$.

For $i \in [1,m]$, consider two edges $u_iw_j$ and $u_iw_k$ for any $1 \le j < k \le m$. If $i \le m-1$, then $\varphi(w_k) > \varphi(w_j) > \varphi(u_i) \ge 0$ which implies 
\begin{align*}
    |\varphi(w_j) - \varphi(u_i)| < |\varphi(w_k) - \varphi(u_i)|.
\end{align*}
Since $u_mw_1 \notin E(G)$, if $i = m$ then we need to only consider the case where $2 \le j < k \le n$. Here, we have $\varphi(v_k) > \varphi(w_j) > \varphi(u_m) \ge 0$. Again, it holds that 
\begin{align*}
    |\varphi(w_j) - \varphi(w_m)| < |\varphi(w_k) - \varphi(w_m)|.
\end{align*}

Similarly for $j \in [1,n]$, consider two edges $u_iw_j$ and $u_kw_j$ for any $1 \le i < k \le n$. If $j \ge 2$, then $\varphi(w_j) > \varphi(u_k) > \varphi(u_i)$. It follows that
\begin{align*}
    |\varphi(w_j) - \varphi(u_k)| < |\varphi(w_j) - \varphi(u_i)|.
\end{align*}
If $j = 1$, it is sufficient to only consider $1 \le i < k \le m-1$. It holds that $\varphi(w_j) > \varphi(u_k) > \varphi(u_i)$. Hence, we have
\begin{align*}
    |\varphi(w_1) - \varphi(u_k)| < |\varphi(w_1) - \varphi(u_i)|.
\end{align*}

Therefore, $\varphi$ is an odd graceful coloring of $G$, with $\max(\varphi) = 2m+2n-4$. This implies $\chi_{og}(K_{m,n} -K_{1,1}) \le 2m+2n-4$ which completes the proof of~\eqref{eqn:complete-minus-one}.

Next, we bound $\chi_{og}(K_{m,n}-K_{1,r})$ for any $r\le n$.  We first note that \begin{align*}
       K_{m-1,n}\subseteq (K_{m,n}-K_{1,r})\subseteq (K_{m,n}-K_{1,1}),
   \end{align*} viewed as subgraph inclusions. Applying Lemma~\ref{lem:subgraph}, Theorem~\ref{thm:complete_bipartite} and Equation~\ref{eqn:complete-minus-one}, we obtain \begin{align}\label{eqn:2m+2n-5/4}\begin{split}
     2m+2n-5 &\le \chi_{og}(K_{m-1,n})\\ &\le \chi_{og}(K_{m,n}-K_{1,r}) \\&\le \chi_{og}(K_{m,n}-K_{1,1})\\ &\le 2m+2n-4.
   \end{split}
   \end{align}  The last inequality gives an upper and lower bound for $\chi_{og}(K_{m,n}-K_{1,r})$, which we will improve in the next sections.

\subsubsection{Graphs satisfying the lower bound} Based on~\eqref{eqn:2m+2n-5/4}, we obtain for all $m,n\ge 2,r\le n$ \[
\chi_{og}(K_{m,n}-K_{1,r})\in \{2m+2n-5,2m+2n-4\}
\]
    We now classify all possible triples $(m,n,r)$  with $\chi_{og}(K_{m,n}-K_{1,r})=2m+2n-5$ First, from Theorem~\ref{thm:complete_bipartite}, note that if $K_{m-1,n}\not \cong K_{2s,2s}$ or $K_{2s,2}$ for some $s\ge 1$, then the lower bound of~\eqref{eqn:2m+2n-5/4} can be improved to $2m+2n-4$. Thus, in this case $\chi_{og}(K_{m,n}-K_{1,r})=2m+2n-4$.  

   Therefore, if $\chi_{og}(K_{m,n}-K_{1,r})=2m+2n-5$, then there exists a positive integer $s$ with \[
   (m,n)  \in \{(2s+1,2),(3,2s),(2s+1,2s)\}
   \] We consider the first two possibilities in this section, and the last case in the next section. However, we first handle the case $s=1$, which in this case corresponds to all of these pairs. We may only need to check the case $r=1$, where we easily obtain \[
   \chi_{og}(K_{3,2}-K_{1,1})=6.
   \]
   Therefore, from now  we let $s\ge 2$.
   
First, when $(m,n)=(2s+1,2)$  we prove that for all $r\le 2s$, \begin{equation}\label{eqn:2s+1,2}\chi_{og}(K_{2s+1,2}-K_{1,r})=4s+2=2m+2n-4.\end{equation}
With respect to Lemma~\ref{lem:subgraph}, we only need to prove the assertion when $r=m-1=2s$. Consider the graph $K_{2s+1,2}-K_{2s,1}$ with bipartitions $U=\{u_1,\dots,u_{2s+1}\}$ and $W=\{w_1,w_2\}$, with edge set $\{u_iw_j\colon 1\le i\le 2s,j=1,2\}\cup \{u_{2s+1}w_1\}.$ Assume that there exists a $(4s+1)$-odd graceful labeling $\lambda$ of this graph.

Note that the graph obtained by erasing $u_{2s+1}w_1$ from this graph is isomorphic to $K_{2s,2}$. Hence, from Theorem~\ref{thm:label_complete_bipartite}, we obtain \[
\{\lambda(u_1),\dots,\lambda(u_{2s})\}=\{2,\dots,4s\}, \quad \{\lambda(w_1),\lambda(w_2)\}=\{1,4s+1\}.
\] Now, note that $\lambda(u_{2s+1})$ is even and bigger than $4s$. Hence, $\max(\lambda(U))\ge 4s+2$, which contradicts our assumption and proves~\ref{eqn:2s+1,2}.

Next, when $(m,n)=(3,2s)$ we prove \begin{equation*}
    \chi_{og}(K_{3,2s}-K_{1,r})=\begin{cases}
        4s+1=2m+2n-5 \:& \text{for}\:2\leq r<m,\\
        4s+2=2m+2n-4 \:& \text{for}\:r=1.\\
    \end{cases}
\end{equation*}

With respect to Lemma~\ref{lem:subgraph}, we only need to consider the case $r=1,2$. First, for $r=2$, we consider the graph $K_{3,2s}-K_{1,2}$ with bipartitions $U=\{u_1,u_2,u_3\}$ and $W=\{w_1,\dots,w_{2s}\}$, with edge set $\{u_iw_j\colon i=1,2,3, 1\le j\le 2s\} - \{u_3w_1,u_3w_{2s}\}.$ We construct a $(4s+1)$-odd graceful labeling $\varphi$ of this graph as follows: \begin{equation}\label{eqn:K3,2s-K1,r}
    \varphi(w_i)=2i,\:1\le i\le m,\quad \varphi(u_1)=1,\quad \varphi(u_2)=4s+1, \quad \varphi(u_3)=3.
\end{equation} Proving  this labeling is an odd graceful labeling is similar to~\eqref{eqn:labeling-2s,2}, except that we also need to consider labelings related to $u_3$. In this case, since $\varphi(w_i)>\varphi(u_3)$ for all $i\ge 2$, we see that all edges adjacent to $u_3$ have pairwise different labelings. Next, with respect to the vertices $w_i$ for $i\ge 2$, we need to compare the value of $|\varphi(w_i)-\varphi(u_j)|$ for $j=1,2,3$. We trivially have \[
|\varphi(w_i)-\varphi(u_1)|=2i-1,\quad |\varphi(w_i)-\varphi(u_2)|=4s+1-2i, \quad |\varphi(w_i)-\varphi(u_3)|=2i-3,
\] which are three pairwise different integers for $1<i\le 2s$, except when $i=s+1$, which does not happen since $u_3$ is not adjacent to $w_{s+1}$. This completes the proof of~\eqref{eqn:K3,2s-K1,r} for $r\ge 2$.

For the case $r=1$, we use contradictions and assume $\lambda$ as a $(4s+1)$-odd graceful labeling of this graph. Let the graph has bipartitions defined as in previous argument, only that in this case we have $u_3w_{j}\notin E$, for some $1\le j\le 2s$. In this setup, since the graph made by removing $u_3$ is a $K_{2,2s}$-complete graph, from Theorem~\ref{thm:label_complete_bipartite}, we obtain, without loss of generality, \[
(\lambda(u_1),\lambda(u_2))=(1,4s+1), \quad (\lambda(w_1),\dots,\lambda(w_{2s}))=(2,\dots,4s).
\]
It remains to label $\lambda(u_3)$, which is an odd number less than $4s+1$ but more than $1$. Note that from Lemma~\ref{lem:abcgraceful}, \[
1+\lambda(u_3)\ne \lambda(w_i)=2i \iff \lambda(u_3)\ne 2i-1
\] {for all} $1\le i\le 2s$, $i\ne j$. Therefore, we have $\lambda(u_3)=2j-1$. However, if $j>2$, \[
\lambda(w_{j-2})+\lambda(w_{j+1})=4y-2=2\lambda(u_3),
\] which contradicts Lemma~\ref{lem:abcgraceful}. 
Therefore, $j=2$, which implies  $\lambda(u_3)=3$. However, in this case, we have \[
\lambda(u_3)+\lambda(u_2)=4s+4=2\lambda(w_{s+1}),
\]which contradicts Lemma~\ref{lem:abcgraceful}. Therefore, the initial assumption is incorrect and the proof of~\eqref{eqn:K3,2s-K1,r} is complete.

\subsubsection{The case $(m,n)=(2s+1,2s)$}
 It remains to prove Theorem~\ref{thm:near-complete} for the case $(m,n)=(2s+1,2s)$. In this section, we prove \begin{equation}\label{eqn:2s+1,2s_other}
   \chi_{og}(K_{2s+1,2s}-K_{1,r}) = \begin{cases}
       8s-3=2m+2n-5,\quad \text{when} \quad s\le r\le 2s-1,\:\text{and}\\
       8s-2=2m+2n-4,\quad \text{when} \quad 1\le r\le s-2.
   \end{cases}
   \end{equation}
   In the remaining case $r=s-1$, we prove\begin{equation}\label{eqn:2s+1,2s_s-1}
   \chi_{og}(K_{2s+1,2s}-K_{1,s-1}) = \begin{cases}
       8s-3,\quad \text{when} \: s\ge 5\: \text{is odd, and}\\
       8s-2,\quad \text{when} \: s\: \text{is even or}\:s=3.
   \end{cases}
   \end{equation}

   These equations are established from the following statements, which we will prove in this section:\begin{itemize}
       \item the existence of an $(8s-3)$-odd graceful labeling of $K_{2s+1,2s}-K_{1,s}$ for odd $s\ge 5$ and  $K_{2s+1,2s}-K_{1,s-1}$ for even $s$ and $s=3$ (sufficient from Lemma~\ref{lem:subgraph}),
       \item a non-existence proof for such labeling for the graph $K_{2s+1,2s}-K_{1,s-2}$ for all $s\ge 5$, and \item a non-existence proof for such labeling for the graph $K_{2s+1,2s}-K_{1,s-1}$ for even $s$ and $s=3$.
   \end{itemize}

%%%%

We first prove the first statement where $s \ne 3$. Let $G = K_{2s+1,2s}-K_{1,r}$ where $r = 2 \lfloor \frac{s}{2} \rfloor$  with bipartitions $W=\{w_1,\dots,w_{2s+1}\}$ and $U=\{u_1,\dots,u_{2s}\}$ such that the graph without $w_{2s+1}$ is isomorphic to $K_{2s,2s}$ and
\begin{align*}
    N(w_{2s+1}) = \{u_1\}\cup \{u_{\lfloor s/2 \rfloor+2},u_{\lfloor s/2 \rfloor+3},\ldots,u_{s+1}\} \cup \{u_{\lfloor\frac{3s}{2}\rfloor+2},u_{\lfloor\frac{3s}{2}\rfloor+3},\ldots,u_{2s}\}.
\end{align*} 
Define a labeling $\psi : V(G) \to [1,8s-3]$ where
\begin{align*}
    \psi(u_i) & = \begin{cases}
        8i - 6, & 1 \le i \le s,\\
        8(i-s) - 4, & s+1 \le i \le 2s.
    \end{cases}\\
    \psi(w_i) & = \begin{cases}
        4i - 3, & 1 \le i \le 2s,\\
        11, & i = 2s+1.
    \end{cases}
\end{align*}
We will show that this is a $(8s-3)$-odd labeling. Observe that $\psi|_X = \varphi$ as defined in \eqref{eqn:labeling-2s,2s} which implies it is sufficient to only show that \begin{align}\label{eqn:psi_ineq_ui}
    |\psi(w_{2s+1}) - \psi(u_i)| \ne |\psi(w_j) - \psi(u_i)|
\end{align}
for $i \in \{1\} \cup [\lfloor s/2 \rfloor+2,s+1]\cup [\lfloor \frac{3s}{2} \rfloor + 2,2s]$, $j \in [1,2s]$ and
\begin{align}\label{eqn:psi_ineq_w2s+1}
    |\psi(w_{2s+1}) - \psi(u_i)| \ne |\psi(w_{2s+1}) - \psi(u_k)|
\end{align}
for $i,k \in \{1\} \cup [\lfloor s/2 \rfloor+2,s+1]\cup [\lfloor \frac{3s}{2} \rfloor + 2,2s]$, $i \ne k$.

First, we will show Equation \eqref{eqn:psi_ineq_ui}. If $i = 1$, then $|\psi(w_{2s+1}) - \psi(u_1)| = |11-2| = 9$. If we assume that $|\psi(w_{2s+1}) - \psi(u_1)| = |\psi(w_j) - \psi(u_1)|$, it follows that \begin{align*}
    |\psi(w_j) - \psi(u_1)| = |\psi(w_j) - 2| = 9.
\end{align*}
Since $\psi(w_j) \ge 2$ for $j \in [1,2s]$, we obtain $\psi(w_j) - 2 = 9$ which implies $\psi(w_j) = 11$. This shows $|\psi(w_{2s+1}) - \psi(u_1)| \ne |\psi(w_j) - \psi(u_1)|$ for every $j \in [1,2s]$. Next, if $i = s+1$ then $|\psi(w_{2s+1}) - \psi(u_{s+1})| = |11-4| = 7$. If $j = 1$ then
\begin{align*}
    |\psi(w_1) - \psi(u_{s+1})| = |1-4| = 3 \ne |\psi(w_{2s+1}) - \psi(u_{s+1})|.    
\end{align*}
Now, for $j \ge 2$ assume that $|\psi(w_{2s+1}) - \psi(u_{s+1})| = |\psi(w_j) - \psi(u_{s+1})|$. Since $\psi(w_j) > 4$ for $j \in [2,2s]$, we have $\psi(w_j) - 4 = 7$ which implies $\psi(w_j) = 11$. This shows that $|\psi(w_{2s+1}) - \psi(u_i)| \ne |\psi(w_j) - \psi(u_i)|$ for $i \in \{1,s+1\}$.

Now, let $i \in [\lfloor s/2 \rfloor+2,s] \cup [\lfloor\frac{3s}{2}\rfloor+2,2s]$. Since 
\begin{align*}
    \psi(u_i) \ge 8(\lfloor s/2\rfloor+2) - 6 > 11 
\end{align*}
for any $i \in [\lfloor s/2 \rfloor+2,s] \cup [\lfloor\frac{3s}{2}\rfloor+2,2s]$, $|\psi(w_{2s+1}) - \psi(u_i)| = \psi(u_i) - 11$. Moreover, if $\psi(w_j) < \psi(u_i)$ then $|\psi(w_j) - \psi(u_i)| = \psi(u_i) - \psi(w_j)$. If we assume that $|\psi(w_{2s+1}) - \psi(u_i)| = |\psi(w_j) - \psi(u_i)|$, then $\psi(u_i) - \psi(w_j) = \psi(u_i) - 11$. This implies $\psi(w_j) = 11$. If $\psi(w_j) \ge \psi(u_i)$, since $\psi(w_j) \le 8s-3$, we obtain
\begin{align}\label{eqn:upper_bound_psi}\begin{split}
    |\psi(w_j) - \psi(u_i)| & = \psi(w_j) - \psi(u_i)\\
    & \le (8s-3) - (8(\lfloor s/2\rfloor+2)-6)\\
    & \le 8(s - \lfloor s/2 \rfloor ) - 13\\
    & \le 4s - 9
\end{split}\end{align}
and
\begin{align}\label{eqn:lower_bound_psi}\begin{split}
    |\psi(u_i) - \psi(w_{2s+1})| & \ge 8(\lfloor s/2\rfloor+2) - 6 - 11\\
    & \ge 8\lfloor s/2 \rfloor - 1\\
    & \ge 4s - 5.
\end{split}\end{align}
By combining Equation \eqref{eqn:upper_bound_psi}-\eqref{eqn:lower_bound_psi}, we have $|\psi(u_{2s+1}) - \psi(w_i)| > |\psi(u_j) - \psi(w_i)|$. This shows Equation \eqref{eqn:psi_ineq_ui}. To show Equation \eqref{eqn:psi_ineq_w2s+1}, observe that
\begin{align}\label{eqn:center_w2s+1}\begin{split}
    &|\psi(w_{2s+1}) - \psi(u_{1})|  =9,\\
    &|\psi(w_{2s+1}) - \psi(u_{s+1})|  =7,\\
    &|\psi(w_{2s+1}) - \psi(u_{i})|  = 8i-17, \quad \text{if }i \in [\lfloor s/2 \rfloor+2,s],\:\text{and}\\
    &|\psi(w_{2s+1}) - \psi(u_{i})|  = 8(i-s)-15,  \quad \text{if }i \in [\lfloor3s/2\rfloor+2,2s].\\
\end{split}\end{align}
If $s = 2$, then $N(w_{2s+1}) = \{u_1,u_{s+1}\}$ and Equation \eqref{eqn:center_w2s+1} implies Equation \eqref{eqn:psi_ineq_w2s+1}. Otherwise, let $s \ge 4$ and let $i \in [\lfloor s/2 \rfloor+2,s] \cup [\lfloor\frac{3s}{2}\rfloor+2,2s]$. From Equation \eqref{eqn:lower_bound_psi}, we obtain $|\psi(w_{2s+1}) - \psi(u_{i})| \ge 4(4) - 5 = 11$. Since
\begin{align*}
   & |\psi(w_{2s+1}) - \psi(u_{i})|  \equiv 3 \pmod 4\quad  \text{if }i \in [\lfloor s/2 \rfloor+2,s] \text{ and}\\
   & |\psi(w_{2s+1}) - \psi(u_{i})|  \equiv 1 \pmod 4  \quad \text{if }i \in [\lfloor3s/2\rfloor+2,2s],
\end{align*}
and $|\psi(w_{2s+1}) - \psi(u_{i})|$ is strictly monotone, it follows that Equation \eqref{eqn:psi_ineq_w2s+1} holds. This shows that $\psi$ is a $(8s-3)$-odd labeling whenever $s \ne 3$. To show the case where $s = 3$ observe the graph $K_{7,6}-K_{1,3}$ along with its $21$-odd graceful coloring in Figure \ref{fig:K76-K13}. This completes the first part of the proof.

\begin{figure}[h!]
    \centering
    \includegraphics[scale=0.6]{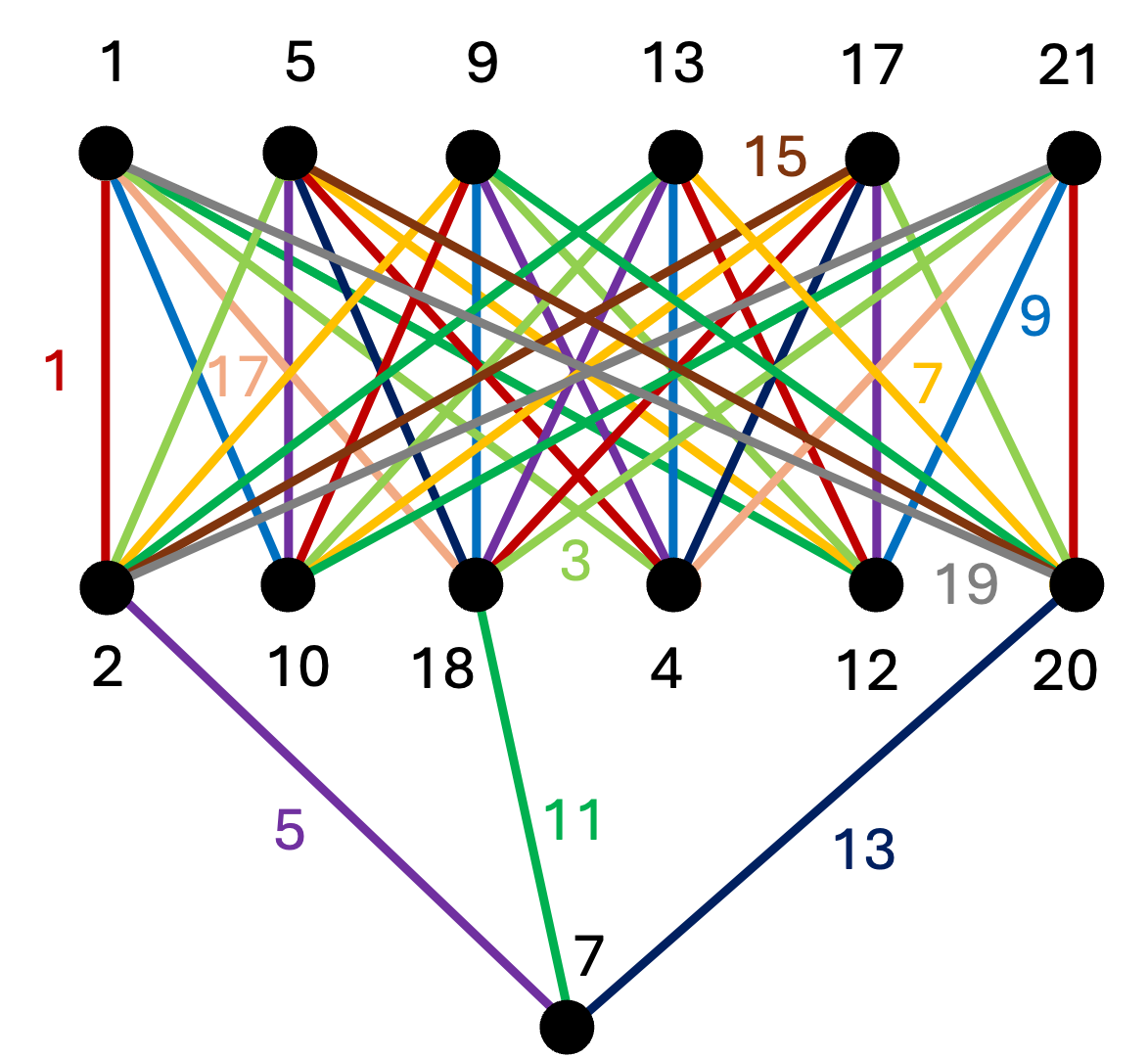}
    \caption{A 21-odd graceful coloring of $K_{7,6}-K_{1,3}$.}
    \label{fig:K76-K13}
\end{figure}

Next, we prove the nonexistence of an $(8s-3)$-odd graceful labeling on the graph $K_{2s+1,2s}-K_{1,r}$ in the cases  mentioned in~\eqref{eqn:2s+1,2s_other} and~\eqref{eqn:2s+1,2s_s-1}.

 Consider the graph $K_{2s+1,2s}-K_{1,r}$ with bipartitions  $U=\{u_1,\dots,u_{2s+1}\}$, $W=\{w_1,\dots,w_{2s}\}$ such that the graph without $u_{2s+1}$ is isomorphic to $K_{2s,2s}$, and suppose there exists an $(8s-3)$-odd graceful coloring $\lambda$ of this graph.

   Note that from Theorem~\ref{thm:label_complete_bipartite}, we obtain either (without loss of generality) \begin{align}\label{eqn:U odd W even}\begin{split}\lambda(u_i) & = 4i - 3, \quad 1 \le i \le 2s.\\
   \lambda(w_i) & = \begin{cases}
        8i - 6, & 1 \le i \le s,\\
        8(i-s) - 4, & s+1 \le i \le 2s,
    \end{cases}
   \end{split}
   \end{align}
   or
   \begin{align}\label{eqn:U even W odd}\begin{split}                       
   \lambda(u_i) & = \begin{cases}
        8i - 6, & 1 \le i \le s,\\
        8(i-s) - 4, & s+1 \le i \le 2s,
    \end{cases}\\
    \lambda(w_i) & = 4i - 3, \quad 1 \le i \le 2s.\\
   \end{split}
   \end{align} It remains to label $u_{2s+1}$. In the next arguments, we prove that if an $(8s-3)$-odd graceful labeling of $K_{2s+1,2s}-K_{1,r}$ exists, 
   \begin{align*}
       \deg u_{2s+1} \le s+1.
   \end{align*} This inequality directly implies $r\ge s-1$, which proves~\eqref{eqn:2s+1,2s_other}.
   
   To prove this, observe that all neighbors of $u_{2s+1}$ are in $W$, and each vertices in $W$ are adjacent to $u_i$, for $1\le i \le n$. Therefore, from Lemma~\ref{lem:abcgraceful}, if $w_j\in N(u_{2s+1})$,  for all $1\le i \le 2s$. \begin{align*}
   \lambda(u_i) \neq 2\lambda(w_j) -\lambda(u_{2s+1}).
   \end{align*}

Consider the first case,~\eqref{eqn:U odd W even}. In this case, note that $\lambda(u_{2s+1})\equiv 3\pmod 4$. Now, let \begin{align}\label{eqn:U odd W even mod 16}
    \lambda(u_{2s+1})=16t-q,\quad 1\le t \le \lceil s/2\rceil,\quad
    q\in \{1,5,9,13\}.
\end{align} Note that all possible values for $\lambda(u_{2s+1})$ are contained in this parametrization.

We now prove the following proposition. \begin{proposition}
    Let $\lambda$ be an odd graceful labeling of $K_{2s+1,2s}-K_{1,r}$ satisfying~\eqref{eqn:U odd W even} and~\eqref{eqn:U odd W even mod 16}. Then, we have $u_{2s+1}w_{j}\notin E(K_{2s+1,2s}-K_{1,r})$ if $\lambda(w_j)$  lies in the following interval, based on the value of $q$: \begin{table}[h]
\begin{tabular}{|c|c|}
\hline
$q$   & \text{$\lambda(w_j)$}           \\ \hline
$1$        & $[8t+2,8t+4s-2]$  \\ \hline
$5$          & $[8t+2,8t+4s-4]$  \\ \hline
$9$         &  $[8t-4,8t+4s-6]$  \\ \hline
$13$        & $[8t-6,8t+4s-8]$\\ \hline
\end{tabular}
\end{table}
\end{proposition}

\begin{proof}
    We only prove the case $q=1$. Let $\lambda(w_j)=8t+2p$, for $1\le p\le 2s-1$. Then, \[
    2\lambda(w_j)-\lambda(u_{2s+1})=2(8t+2p)-(16t-1)=4p+1 = \lambda(u_{p+1}),
    \]which completes the proof for this case. The proofs for the other cases of $q$ follow similar arguments.
\end{proof}

However, we note that all numbers in $\lambda(W)$ is either $2$ or $4$ modulo $8$. By counting carefully the numbers in each of the intervals that are $2$ or $4$ modulo $8$ in these intervals, we obtain the following corollary. \begin{corollary}\label{cor:label mod 16}
    Let $\lambda$ be an odd graceful coloring of $K_{2s+1,2s}-K_{1,r}$ satisfying~\eqref{eqn:U odd W even} and~\eqref{eqn:U odd W even mod 16}. Then, the largest possible value of $\deg(u_{2s+1})$ is determined in the following table, based on the value of $q$ and $s\ge 4$.\begin{table}[ht]
\begin{tabular}{|c|c|c|}
\hline
$q\slash s$ & \text{odd}              & \text{even}              \\ \hline
$1$         & $s$ & $s$  \\ \hline
$5$         & $s+1$ & $s$  \\ \hline
$9$         & $s$ & $s$  \\ \hline
$13$        & $s-1$ & $s$ \\ \hline
\end{tabular}
\end{table}
\end{corollary}
\begin{proof}
 Here we only prove the case $q=5$ and odd $s$. Note that the smallest and largest number that are $2$ modulo $8$ and lies inside the interval $[8t+2,8t+4s-4]$ are $8t+2$ and $8t+4s-10$, respectively. Hence, there are $(s-1)/2$ numbers in this interval that are $2$ modulo $8$. 

 Next, the smallest and largest number that are $4$ modulo $8$ and lies inside the interval $[8t+2,8t+4s-4]$ are $8t+4$ and $8t+4s-8$, respectively. Hence, there are $(s-1)/2$ numbers in this interval that are $4$ modulo $8$. 

 These implies there are $s-1$ vertices in $W$ that are not adjacent to $u_{2s+1}$, which implies $\deg(u_{2s+1})\le s+1$. We repeat this argument for other values of $q$ and also for even $s$, which completes the proof.
\end{proof}

Since in this case $\deg(u_{2s+1})\le s+1$ for all $s$, this implies there does not exist an $(8s-3)$-odd graceful labeling $\lambda$ of $K_{2s+1,2s}-K_{1,s-2}$ that satisfies \eqref{eqn:U odd W even} for all $s$. In addition, when $s$ is even, there does not exist an $(8s-3)$-odd graceful labeling $\lambda$ of $K_{2s+1,2s}-K_{1,s-1}$ that satisfies \eqref{eqn:U odd W even}. 

Now, let $\lambda$ be an $(8s-3)$-odd graceful labeling of $K_{2s+1,2s}-K_{1,r}$ satisfying~\eqref{eqn:U even W odd}. In this case, note that $\lambda(u_{2s+1})\equiv 6,0\pmod{8}$. Therefore, we may write \begin{equation*}
    \lambda(u_{2s+1})=8t-q, \quad 1\le t< s,\quad q\in \{0,2\}.
\end{equation*}
In this case, we claim the following proposition.
\begin{proposition}\label{prp:4t+4s-3}
    Let $\lambda$ be an odd graceful labeling of $K_{2s+1,2s}-K_{1,r}$ satisfying~\eqref{eqn:U odd W even} and~\eqref{eqn:U odd W even mod 16}. Then, we have $u_{2s+1}w_{j}\notin E(K_{2s+1,2s}-K_{1,r})$ if $\lambda(w_j)$  lies between the interval $[4t+1,4t+4s-3]$. 

\end{proposition}
\begin{proof}
    Let $\lambda(w_j)=4t+4p-3$, for $1\le p \le s$. Then, \[
    2\lambda(w_j)-\lambda(u_{2s+1})=2(4t+4p-3)-(8t-q)=8p+q-6.
    \] Since $q$ is either $0$ or $2$, this implies the existence of an index $j'$ with $\lambda(w_{j'})=8p+q-6$. Therefore, $u_{2s+1}w_j\notin E(K_{2s+1,2s}-K_{1,r})$, which completes the proof. 
\end{proof}
There are exactly $s$ integers that are $1$ modulo $4$ in the interval. Hence, \begin{equation}\label{eqn:u(2s+1) even}
    \deg(u_{2s+1})\le s
\end{equation} in this labeling case, which implies there are no $(8s-3)$-odd graceful labeling $\lambda$ of $K_{2s+1,2s}-K_{1,s-1}$ that satisfies~\eqref{eqn:U even W odd}.

Therefore, from Proposition~\ref{cor:label mod 16} and~\ref{eqn:u(2s+1) even}, we obtain for $s\ge 4$, \begin{align*}
    &\chi_{og}(K_{2s+1,2s}-K_{1,s-2})=8s-2,\\
    &\chi_{og}(K_{2s+1,2s}-K_{1,s-1}) = \begin{cases}
       8s-3,\quad \text{when} \: s\ge 5\: \text{is odd, and}\\
       8s-2,\quad \text{when} \: s\: \text{is even}.
       \end{cases}
\end{align*}

It remains to compute $\chi_{og}(K_{2s+1,2s}-K_{1,s-1})$ for $s=3$. Assume that an $(8s-3)=21$-odd graceful labeling exist. From Corollary~\ref{cor:label mod 16}, note that $\lambda(u_7)=11$, and $u_7$ is adjacent to exactly four vertices of $W$, with \[
\lambda(W)=\{2,4,10,12,18,20\}.
\]However, Proposition~\ref{prp:4t+4s-3} implies $u_7$ is not adjacent to vertices with label $10$ and $12$. Therefore, $u_7$ is adjacent to vertices whose labels are $2,4,18,20$. In particular, we recall $\lambda(w_1)=2$, $\lambda(w_6)=20$. However \[
\lambda(w_1)+\lambda(w_6)=22=2\lambda(u_7),
\] which contradicts Lemma~\ref{lem:abcgraceful}. Therefore, \[
\chi_{og}(K_{7,6}-K_{1,2})=22,
\] which completes our classification. This completes the proof of~\eqref{eqn:2s+1,2s_other} and~\eqref{eqn:2s+1,2s_s-1}, which also completes the proof of Theorem~\ref{thm:near-complete}.


\begin{thebibliography}{99}

\bibitem{BBELZ} Z. Bi, A. Byers, S. English, E. Laforge, P. Zhang, `Graceful colorings of graphs', \textit{J. Combin. Math. Combin. Comput}, \textbf{101} (2017), 101-119.

\bibitem{Brooks} R. L. Brooks, `On colouring the nodes of a network', \textit{Math. Proc. Camb. Philos. Soc.}, \textbf{37} 2 (1941), 194–197. 

\bibitem{CL} G. Chartrand and L. Lesniak, \textit{Graphs \& digraphs}, Chapman and Hall, CRC,
4th edition, (2005).

\bibitem{Die} R. Diestel, \textit{Graph theory}, Springer Nature, 5th edition, (2017).

\bibitem{Gal} J. A. Gallian, `Graph labeling: dynamic survey (DS6)', \textit{Electron. J. Combin.} \textbf{5} (2024).

\bibitem{Gnana} R. B. Gnanajothi, \textit{Topics in graph theory}, Ph.D. thesis, Madurai Kamaraj University (1991).

\bibitem{SSY} J. Su, H. Sun, and B. Yao, `Odd-graceful total colorings for constructing graphic lattice', \textit{Mathematics} \textbf{10} 1 (2022).

\bibitem{SLHB} I. N. Suparta , Y. Lin , R. Hasni and I. N. Budayana, `On odd-graceful coloring of graphs', \textit{Commun. Comb. Optim.}, \textbf{10} 2 (2025), 335-354.


\bibitem{TV} T. Tao and V. H. Vu, \textit{Additive combinatorics}, Cambridge University Press, (2006).
\end{thebibliography}
\end{document}